\documentclass{amsart}
\usepackage[utf8]{inputenc}
\usepackage{tikz-cd}
\usepackage{amsmath}
\usepackage{verbatim}
\usepackage{amssymb}
\usepackage{amsthm}
\usepackage{hyperref}
\usepackage[bb=dsserif]{mathalpha}

\newtheorem*{theorem*}{Theorem}
\newtheorem*{proposition*}{Proposition}

\usepackage{ulem}
\usetikzlibrary{calc}
\usetikzlibrary{decorations.pathmorphing}
\tikzset{curve/.style={settings={#1},to path={(\tikztostart)
    .. controls ($(\tikztostart)!\pv{pos}!(\tikztotarget)!\pv{height}!270:(\tikztotarget)$)
    and ($(\tikztostart)!1-\pv{pos}!(\tikztotarget)!\pv{height}!270:(\tikztotarget)$)
    .. (\tikztotarget)\tikztonodes}},
    settings/.code={\tikzset{quiver/.cd,#1}
        \def\pv##1{\pgfkeysvalueof{/tikz/quiver/##1}}},
    quiver/.cd,pos/.initial=0.35,height/.initial=0}

\tikzset{tail reversed/.code={\pgfsetarrowsstart{tikzcd to}}}
\tikzset{2tail/.code={\pgfsetarrowsstart{Implies[reversed]}}}
\tikzset{2tail reversed/.code={\pgfsetarrowsstart{Implies}}}
\tikzset{no body/.style={/tikz/dash pattern=on 0 off 1mm}}

\DeclareMathOperator{\im}{im}

\DeclareMathOperator{\field}{k}
\DeclareMathOperator{\op}{op}
\DeclareMathOperator{\rad}{rad}

\DeclareMathOperator{\modu}{mod}
\DeclareMathOperator{\Ext}{Ext}

\DeclareMathOperator{\End}{End}
\DeclareMathOperator{\Hom}{Hom}

\DeclareMathOperator{\Sim}{Sim}

\DeclareMathOperator{\id}{id}

\DeclareMathOperator{\filt}{F}
\newtheorem{theorem}{Theorem}[section]
\newtheorem{definition}[theorem]{Definition}
\newtheorem{example}[theorem]{Example}
\newtheorem{corollary}[theorem]{Corollary}
\newtheorem{lemma}[theorem]{Lemma}
\newtheorem{remark}[theorem]{Remark}
\newtheorem{proposition}[theorem]{Proposition}
\title{Exact Borel Subalgebras of Tensor Algebras of Quasi-Hereditary Algebras}
\author{Anna Rodriguez Rasmussen}

\begin{document}
\bibliographystyle{plain}
\begin{abstract}
Given two quasi-hereditary algebras, their tensor product is quasi-hereditary. In this article, we show that given two exact Borel subalgebras for these quasi-hereditary algebras, their tensor product is an exact Borel subalgebra. Moreover, we describe in which cases the tensor product of two regular exact Borel subalgebras is again regular. Additionally, we investigate tensor algebras of generalised species of quasi-hereditary algebras and exact Borel subalgebras thereof.
\end{abstract}
\maketitle
\tableofcontents
\section{Introduction}

Quasi-hereditary algebras, first defined by Cline, Parshall and Scott \cite{CPS}, appear in many different areas of representation theory. Some examples include Schur algebras, algebras of global dimension at most two and algebras underlying blocks of category $\mathcal{O}$. Additionally, for non-quasi-hereditary algebras, it is sometimes possible to associate quasi-hereditary covers, a concept introduced by Rouquier \cite{Rouquier}.\\ 
In \cite{Koenig}, König defined the concept of an exact Borel subalgebra of a quasi-hereditary algebra. An exact Borel subalgebra $B$ of a quasi-hereditary algebra $A$ is a subalgebra $B\subseteq A$ capturing the homological information of the category $\filt(\Delta)$ of standardly filtered modules.
As such, an exact Borel subalgebra $B$ of $A$ is a helpful tool to study the quasi-hereditary structure of $A$. Of particular interest among all exact Borel subalgebras are so-called regular exact Borel subalgebras, which exhibit desirable homological properties, and for which both existence and uniqueness results are known \cite{KKO, Miemietz, uniqueness}.\\
In the past, it has been established that a quasi-hereditary structure is compatible with many constructions on algebras. For example, if $A$ and $A'$ are quasi-hereditary then, by \cite{Chan}, so is their tensor product $A\otimes A'$. Similarly, if $G$ is a finite group acting on a quasi-hereditary algebra $A$,  such that the characteristic of the underlying field does not divide the order of $G$, then, under a suitable compatibility condition, the skew group algebra $A*G$ is quasi-hereditary. In particular, wreath products of quasi-hereditary algebras, which appear in work by Chuang and Kessar on Broué's Abelian Defect Group Conjecture \cite{chuangkessar}, are quasi-hereditary.\\ Moreover, it was shown in \cite{triangular_matrix} that if $A$ and $A'$ are quasi-hereditary and $M$ is a left-standardly filtered $A$-$A'$-bimodule, then the associated triangular matrix ring $\begin{pmatrix}
    A & M \\ 0  & A'
\end{pmatrix}$ is quasi-hereditary.\\
Given these results, it is a natural question how and under which conditions one may construct an exact Borel subalgebra of the new algebra based on exact Borel subalgebras of the original algebras, and in which cases regularity is preserved.
In past work, the authors have investigated the case of skew group algebras \cite{skewgroup}.
The aim of the present article is to consider the same question for two other constructions of this kind. In the first half of the article, we consider tensor products of quasi-hereditary algebras, and show the following result:
\begin{theorem*}\ref{thm_tensor} \ref{proposition_regular}
 Let $A, A'$ be quasi-hereditary algebras with exact Borel subalgebras $B, B'$. Then $B\otimes B'$ is an exact Borel subalgebra of $A\otimes A'$.\\
 Moreover, suppose that $B$ and $B'$ are regular. Then the following are equivalent: 
		\begin{enumerate}
			\item $B\otimes B'$ is regular.
            \item One of the following statements holds:
            \begin{enumerate}
			\item $A$ and $A'$ are directed.
			\item $A^{\op}$ and $(A')^{\op}$ are directed.
			\item $A$ is semisimple.
            \item $A'$ is semisimple.
            \end{enumerate}
		\end{enumerate}
\end{theorem*}
In the second half of the article, we consider tensor algebras of species of quasi-hereditary algebras.
Species are a generalization of path algebras of quivers; in our generalized setting, they consist simply of an underlying quiver, an algebra at every vertex and a bimodule at every arrow. However, the original versions of species are usually more restrictive, in particular, the algebras at the vertices are often assumed to be division rings. In this form, species appear in the classification of hereditary algebras over perfect fields in \cite{Gabriel}, and fulfill a version of Gabriel's theorem \cite{DR3}. However, more general versions of species are also commonly studied \cite{julian_basic, coelholiu, Lemay}. Species are related to tensor products of algebras in that the tensor product of an algebra $A$ with a path algebra can be viewed as a tensor algebra over the species corresponding to the quiver $Q$ equipped with the algebra $A$ at every vertex and the $A$-$A$-bimodule $A$ at every arrow. Moreover, they generalize triangular matrix rings in that a triangular matrix ring is nothing but a species on the $\textup{A}_2$-quiver.\\
Here, we first establish a criterion for quasi-heredity, Theorem \ref{thm_qh_species}, and thereafter give a construction for an exact Borel subalgebra in Theorem \ref{thm_borel_species}. The assumptions as well as the construction itself are rather technical. However, both simplify significantly in the case of triangular matrix rings: 
 \begin{theorem*}\ref{proposition_triangular_borel} (see also \ref{thm_borel_species})
    Suppose $A_1$ and $A_2$ are quasi-hereditary algebras and $M$ is a left standardly filtered $A_2$-$A_1$-bimodule. Let $A$ be the triangular matrix ring 
\begin{align*}
    \begin{pmatrix}
        A_2 & M\\
        0 & A_1
    \end{pmatrix}
\end{align*}
    Suppose that $B_1$ and $B_2$ are exact Borel subalgebras of $A_1$ resp. $A_2$ with embeddings $\iota_1$ resp. $\iota_2$, and $M\cong A_2\otimes_{B_2} N$ as an $A_2$-$B_1$-bimodule for some $N_2$-$N_1$-bimodule $B$. Then the triangular matrix ring $B$ given by
    \begin{align*}
    \begin{pmatrix}
        B_2 & N\\
        0 & B_1
    \end{pmatrix}
\end{align*}
is an exact Borel subalgebra of $A$ via the embedding
     \begin{align*}
           \begin{pmatrix}
        \iota_2 & f_N\\
        0 & \iota_1
    \end{pmatrix}: B\rightarrow A.
     \end{align*}
 \end{theorem*}
The structure of the article is a follows:\\
In Section 2, we establish our notation and give a short introduction to quasi-hereditary algebras.\\
In Section 3, we investigate tensor products of quasi-hereditary algebras. We establish in Theorem \ref{thm_tensor} that tensor products of exact Borel subalgebras are exact Borel subalgebras. Moreover, we investigate regularity of tensor products of regular exact Borel subalgebras, see Proposition \ref{proposition_regular}.\\
In Section 4, we consider tensor algebras of species of quasi-hereditary algebras. We begin by giving a short introduction to species, and establishing a sufficient criterion for quasi-heredity, which is based on the criterion in \cite{triangular_matrix}.
Afterwards, we give a construction for an exact Borel subalgebra of the species based on a collection of exact Borel subalgebras of the underlying collection of algebras, under some technical conditions on the bimodules. We give some examples where this construction simplifies, including the case of triangular matrix rings.
Finally, we conclude with some examples to illustrate the question of regularity.
\section{Preliminaries}
Throughout, let $\field$ be an algebraically closed field. All $\field$-vector spaces we consider are finite-dimensional, and all tensor products, unless otherwise stated, are over $\field$. For any subset $S$ of a vector space $V$ we denote by $\langle S\rangle $ the $\field$-span of $S$.\\
All algebras we consider are finite-dimensional associative unital $\field$-algebras, and all modules are finitely generated left modules. For any finite-dimensional algebra we denote by $\Sim(A)$ a set of representatives of isomorphism classes of simple $A$-modules. Moreover, if $e$ and $f$ are primitive orthogonal idempotents of $A$, we call $e$ and $f$ equivalent if the associated simple modules $Ae/\rad(A)e$ and $Af/\rad(A)f$ are isomorphic.\\
Quasi-hereditary algebras were first defined by Cline, Parshall and Scott:
\begin{definition}\cite[p.92]{CPS}(Quasi-hereditary, ideal definition)
    Let $A$ be a finite-dimensional algebra. An ideal $J$ is called a \textbf{heredity ideal} in $A$ if 
    \begin{itemize}
        \item $J^2=J$,
        \item $J\rad(A)J=(0)$
        \item $J$ is projective as a right $A$-module.
    \end{itemize}
    The algebra $A$ is called quasi-hereditary if there is a heredity ideal $J$ in $A$ and $A=J$ or $A/J$ is quasi-hereditary.
\end{definition}
In other words, $A$ is quasi-hereditary iff there is a chain of ideals
\begin{align*}
    (0)\subset J_n\subset \dots \subset J_1=A
\end{align*}
in $A$, such that $J_k/J_{k+1}$ is a heredity ideal in $A/J_{k+1}$ for all $1\leq k<n$. 
In \cite[Lemma 3.4]{CPS}, it was also shown that if $A$ is quasi-hereditary, there is a set of orthogonal idempotents $e_1, \dots, e_n$ such that $J_k= A\sum_{i=k}^n e_i A$ for all $1\leq k\leq n$, $\sum_{k=1}^n e_i=1$ and $e_iAe_j\subseteq\rad(A)$ for $i\neq j$. Decomposing each $e_i$ further into a sum of primitive orthogonal idempotents $e_i=\sum_{j=1}^{m_i}e_{ij}$, we can obtain a pre-order $\lesssim_A$ on a complete set of primitive orthogonal idempotents for $A$ by setting $e_{ij}<_A e_{kl}:\Leftrightarrow i< k$ and $e_{ij}\sim_A e_{kl}$ if and only if  $e_{ij}$ is equivalent to $e_{kl}$. Here, it is important to note that $e_{ij}$ being equivalent to $e_{kl}$ implies that $i=k$.\\
Since primitive idempotents give rise to simple $A$-modules, we have for any complete set $S=\{f_1, \dots, f_N\}$ of primitive orthogonal idempotents of $A$ a bijection
\begin{align*}
    S/\textup{equivalence}\rightarrow \Sim(A), [f_i]\mapsto [Af_i/\rad(A)f_i].
\end{align*}
In particular, pre-orders $\lesssim$ on $S$ such that $f_i\sim f_j$ if and only if $f_i$ and $f_j$ are equivalent are in one-to-one correspondence with partial orders $\leq$ on $\Sim(A).$
Hence the pre-order $\lesssim_A$ on $\{e_{ij}|1\leq i\leq n, 1\leq j\leq m_i\}$ induces a partial order $\leq_A$ on $\Sim(A)$.
Using this partial order, it is possible to reconstruct the chain of ideals. Hence quasi-heredity of $A$ with respect to the chain of ideals
\begin{align*}
    (0)\subset J_n\subset \dots \subset J_1=A
\end{align*}
can be reformulated into a criterion on $(A, \leq_A)$:
\begin{definition}\label{definition_qh}\cite[p. 2]{DlabRingel}
    Let $A$ be a finite-dimensional algebra and $\leq_A$ be a partial order on $\Sim(A)$. Then for every $L\in \Sim(A)$ we define
    \begin{align*}
        \Delta(L):=P(L)/\sum_{L'\nleq_A L, f\in \Hom_A(P_{L'}, P_L)}\im(f).
    \end{align*}
    $\Delta(L)$ is called the standard module associated to $L$ with respect to the partial order $\leq_A$.
    Dually, we define $\nabla(L)$ as the biggest submodule of $I(L)$ such that all composition factors $L'$ of $\nabla(L)$ fulfill $L'\leq_A L$.
    We denote by $\Delta:=(\Delta(L))_{L\in \Sim(A)}$ the collection of standard modules and by $\filt(\Delta)$ the full subcategory of $\modu A$ consisting of all modules which admit a filtration by $\{\Delta(L)|L\in \Sim(A)\}$.\\
    Dually, we denote by $\nabla:=(\nabla(L))_{L\in \Sim(A)}$ the collection of costandard modules and by $\filt(\nabla)$ the full subcategory of $\modu A$ consisting of all modules which admit a filtration by $\{\nabla(L)|L\in \Sim(A)\}$.\\
    Then the following statements are equivalent: 
   \begin{enumerate}
   \item  The pair $(A, \leq_A)$ is quasi-hereditary.
       \item $\End_{\field}(\Delta(L))\cong \field$ for all $L\in \Sim(A)$ and $A\in \filt(\Delta)$.
       \item $\End_{\field}(\nabla(L))\cong \field$ for all $L\in \Sim(A)$ and $A\in \filt(\nabla)$.
   \end{enumerate}
\end{definition} 
Here, the partial order $\leq_A$ on $\Sim(A)$ is considered part of the structure of the quasi-hereditary algebra $A$. Indeed, in general $A$ may admit several different partial orders $\leq_A$ which give $(A, \leq_A)$ the structure of a quasi-hereditary algebra, which can have very different standard modules. Since the standard modules and the category $\filt(\Delta)$ are important objects of study in themselves, we want to view these as different quasi-hereditary algebras.\\
On the other hand, it may also happen that different partial orders give rise to the same standard modules. In this case, they are called equivalent partial orders.\\
A particularly easy case of a quasi-hereditary algebra $(A, \leq_A)$ is the case where $\filt(\Delta)=\modu A$, i.e. when the standard modules are simple; in this case $(A, \leq_A)$ is called directed.\\
Inspired by Lie algebras, one can study quasi-hereditary algebras using certain directed subalgebras, called exact Borel subalgebras, whose definition is due to König \cite{Koenig}:
\begin{definition}
    Let $(A,\leq_A)$ be a finite-dimensional algebra. A subalgebra $B$ is called an exact Borel subalgebra if
    \begin{enumerate}
        \item The induction functor  \begin{align*}
        A\otimes_B -:\modu B\rightarrow \modu A
    \end{align*}
    is exact, in other words, $A$ is projective as a right $B$-module.
    \item There is a bijection $\phi:\Sim(B)\rightarrow\Sim(A)$ such that for all $L\in \Sim(B)$ we have
    \begin{align*}
        A\otimes_B L\cong \Delta(\phi(L)).
    \end{align*}
    \item $(B, \leq_B)$ is directed with respect to the partial order
    \begin{align*}
        L\leq_B L':\Leftrightarrow \phi(L)\leq_A\phi(L').
    \end{align*}
    \end{enumerate}
\end{definition}
Exact Borel subalgebras may have additional desirable properties; in the following, we name two of them.
\begin{definition}\label{definition_regular}(\cite[Definition 3.4]{BKK} and \cite[p. 405]{Koenig})
    \begin{enumerate}
        \item An exact Borel subalgebra $B$ of a quasi-hereditary algebra $(A, \leq_A)$ is called strong if it contains a maximal semisimple subalgebra of $A$.
        \item An exact Borel subalgebra $B$ of a quasi-hereditary algebra $(A, \leq_A)$ is called regular if for every $n\geq 1$ the maps
        \begin{align*}
            \Ext^n_B(L^B, L^B)\rightarrow   \Ext^n_A(\Delta^A, \Delta^A), [f]\mapsto [\id_A\otimes_B f]
        \end{align*}
        are isomorphisms, where $L^B=\bigoplus_{L_i^B\in \Sim(B)}L_i^B$ and $\Delta^A=A\otimes_B L^B\cong \bigoplus_{L_i^A\in \Sim(A)}\Delta(L_i^A)$.
        \item An exact Borel subalgebra $B$ of a quasi-hereditary algebra $(A, \leq_A)$ is called homological if the map
        \begin{align*}
            \Ext^1_B(L^B, L^B)\rightarrow   \Ext^1_A(\Delta^A, \Delta^A), [f]\mapsto [\id_A\otimes_B f]
        \end{align*}
        is an epimorphism, and for every $n\geq 2$ the maps
        \begin{align*}
            \Ext^n_B(L^B, L^B)\rightarrow   \Ext^n_A(\Delta^A, \Delta^A), [f]\mapsto [\id_A\otimes_B f]
        \end{align*}
        are isomorphisms.
    \end{enumerate}
\end{definition} 
It is important to note that for regular exact Borel subalgebras, there are existence and uniqueness results \cite{KKO, Miemietz, uniqueness}, while the same is not true for exact Borel subalgebras in general (see \cite[Example 2.3]{Koenig} and \cite[Example 3.12]{monomialborel}).
\section{Tensor Products}
In \cite{Chan}, it was established that the tensor product of two quasi-hereditary algebras $(A,\leq_A)$ and $(A', \leq_{A'})$ is quasi-hereditary.
In this section, we show that in case $A$ and $A'$ admit exact Borel subalgebras $B$ and $B'$, the tensor product $B\otimes B'$ is an exact Borel subalgebra of $A\otimes A'$. We also investigate in which case regularity of $B$ and $B'$ gives rise to regularity of $B\otimes B'$, which, unfortunately, happens only under very strong conditions.
\begin{theorem}(\cite[Section 2]{Chan})\label{thmqh}\\
  Let $(A, \leq_A)$, $(A', \leq_{A'})$ be quasi-hereditary algebras. Then $A\otimes A'$ is quasi-hereditary with the partial order 
  \begin{align*}
    L^A_i\otimes L^{A'}_{i'}\leq L^A_j\otimes L^{A'}_{j'}:\Leftrightarrow L^A_i\leq_A  L^A_j \textup{ and } L^{A'}_{i'}\leq_{A'} L^{A'}_{j'}.
  \end{align*}
  Moreover, the standard module $\Delta_{L_i^A\otimes L_{i'}^{A'}}^{A\otimes A'}$ is given by 
  \begin{align*}
    \Delta_{L_i^A\otimes L_{i'}^{A'}}^{A\otimes A'}\cong \Delta_{L_i^A}^{A}\otimes \Delta_{L_{i'}^{A'}}^{A'}
  \end{align*}
  In particular, if $A$ and  $A'$ are directed, so is $A\otimes A'$.
\end{theorem}
\begin{theorem}\label{thm_tensor}
 Let $A, A'$ be quasi-hereditary algebras with exact Borel subalgebras $B, B'$. Then $B\otimes B'$ is an exact Borel subalgebra of $A\otimes A'$.\\
 Moreover, $B\otimes B'$ is strong if and only if $B$ and $B'$ are.
\end{theorem}
\begin{proof}
  Note that the projective right $B\otimes B'$-modules are exactly given by tensor products of a projective right $B$-module with a projective right $B'$-module. Hence $A\otimes A'$ is projective as a right $B\otimes B'$-module.
  Moreover, $B\otimes B'$ is directed by Theorem \ref{thmqh}.\\
  By assumption there are bijections
  \begin{align*}
    \varphi:\Sim(B)\rightarrow \Sim(A), L_i^B\mapsto L_i^A\\
    \varphi':\Sim(B')\rightarrow \Sim(A'), L_{i'}^{B'}\mapsto L_{i'}^{A'}
  \end{align*}
  such that 
  \begin{align*}
    A\otimes_B L_i^B\cong \Delta_{L_i^A}^A\\
    A'\otimes_{B'}L_{i'}^{B'}\cong \Delta_{L_{i'}^{A'}}^{A'}.
  \end{align*}
  Hence
  \begin{align*}
    \psi: \Sim(B\otimes B')\rightarrow \Sim(A\otimes A'), L_i^B\otimes L_{i'}^{B'}\mapsto L_i^A\otimes L_{i'}^{A'}
  \end{align*}
  is a bijection such that
  \begin{align*}
    (A\otimes A')\otimes_{B\otimes B'}(L_i^B\otimes L_{i'}^{B'})\cong (A\otimes_B L_i^A)\otimes (A'\otimes_{B'} L_{i'}^{B'})\\
    \cong \Delta_{L_i^A}^A\otimes \Delta_{L_{i'}^{A'}}^{A'}\cong \Delta_{L_i^A\otimes L_{i'}^{A'}}^{A\otimes A'}.
  \end{align*}
  Moreover, let $L^B$ be a maximal semisimple subalgebra of $B$, $L^{B'}$ be a maximal semisimple subalgebra of $B'$. Obviously, $\rad(B)\otimes B'+B\otimes \rad(B')\subseteq \rad(B\otimes B')$, so that, since tensor products of semisimple algebras over an algebraically closed field are semisimple, the subalgebra $L^B\otimes L^{B'}$ is a maximal semisimple subalgebra of $B\otimes B'$. If $B$ and $B'$ are strong, then $L^B=L^A$ is also a maximal semisimple subalgebra of $A$ and $L^{B'}=L^{A'}$ is a maximal semisimple subalgebra of $A'$, so that $L^B\otimes L^{B'}=L^A\otimes L^{A'}$ is a maximal semisimple subalgebra of $A\otimes A'$. If, on the other hand, $B$ is not strong, then $L^B$ is not maximal, and thus properly contained $L^B\subsetneq L^A$ in a semisimple subalgebra of $A$. Hence 
  \begin{align*}
    L^B\otimes L^{B'}\subsetneq L^A\otimes L^{B'}
  \end{align*}
  where the latter is a semisimple subalgebra of $A\otimes A'$. In particular $L^B\otimes L^{B'}$ is not a maximal semisimple subalgebra of $A\otimes A'$. Since all maximal semisimple subalgebras of an algebra are of the same dimension, this implies that $B\otimes B'$ is not strong.
\end{proof}
\begin{example}\cite[Example A1]{KKO}
  Let $A$ be the algebra given by 
  % https://q.uiver.app/#q=WzAsMixbMCwwLCIxIl0sWzEsMCwiMiJdLFswLDEsIlxcYWxwaGEiLDAseyJjdXJ2ZSI6LTF9XSxbMSwwLCJcXGJldGEiLDAseyJjdXJ2ZSI6LTF9XV0=
\[\begin{tikzcd}[ampersand replacement=\&]
	1 \& 2
	\arrow["\alpha", curve={height=-6pt}, from=1-1, to=1-2]
	\arrow["\beta", curve={height=-6pt}, from=1-2, to=1-1]
\end{tikzcd}\]
with relations $\alpha\beta=0$ and order $1<2$.
Then
\begin{align*}
	&P_1:=\begin{pmatrix} 1 \\ 2\\ 1 \end{pmatrix}
  &P_2:=\begin{pmatrix} 2 \\ 1\end{pmatrix}\\
	&\Delta_1=L_1=\begin{pmatrix} 1\end{pmatrix}
  &\Delta_2=P_2=\begin{pmatrix} 2\\ 1\end{pmatrix}.\\
\end{align*}
This has strong regular exact Borel subalgebra
% https://q.uiver.app/#q=WzAsMixbMCwwLCIxIl0sWzEsMCwiMiJdLFswLDEsIlxcYWxwaGEiXV0=
\[\begin{tikzcd}[ampersand replacement=\&]
	1 \& 2.
	\arrow["\alpha", from=1-1, to=1-2]
\end{tikzcd}\]
Now consider the tensor product $A\otimes A$ given by the quiver
% https://q.uiver.app/#q=WzAsNCxbMCwwLCIoMSwxKSJdLFsyLDAsIigyLDEpIl0sWzAsMiwiKDEsMikiXSxbMiwyLCIoMiwyKSJdLFswLDEsIlxcYWxwaGFcXG90aW1lcyBcXGlkXzEiLDAseyJjdXJ2ZSI6LTF9XSxbMCwyLCJcXGlkXzFcXG90aW1lcyBcXGFscGhhIiwyLHsiY3VydmUiOjF9XSxbMSwzLCJcXGlkXzJcXG90aW1lc1xcYWxwaGEiLDIseyJjdXJ2ZSI6MX1dLFsyLDMsIlxcYWxwaGFcXG90aW1lcyBcXGlkXzIiLDAseyJjdXJ2ZSI6LTF9XSxbMywxLCJcXGlkXzJcXG90aW1lcyBcXGJldGEiLDIseyJjdXJ2ZSI6MX1dLFsxLDAsIlxcYmV0YVxcb3RpbWVzXFxpZF8xIiwwLHsiY3VydmUiOi0xfV0sWzMsMiwiXFxiZXRhXFxvdGltZXNcXGlkXzIiLDAseyJjdXJ2ZSI6LTF9XSxbMiwwLCJcXGlkXzFcXG90aW1lc1xcYmV0YSIsMix7ImN1cnZlIjoxfV1d
\[\begin{tikzcd}[ampersand replacement=\&]
	{(1,1)} \&\& {(2,1)} \\
	\\
	{(1,2)} \&\& {(2,2)}
	\arrow["{\alpha\otimes \id_1}", curve={height=-6pt}, from=1-1, to=1-3]
	\arrow["{\id_1\otimes \alpha}"', curve={height=6pt}, from=1-1, to=3-1]
	\arrow["{\id_2\otimes\alpha}"', curve={height=6pt}, from=1-3, to=3-3]
	\arrow["{\alpha\otimes \id_2}", curve={height=-6pt}, from=3-1, to=3-3]
	\arrow["{\id_2\otimes \beta}"', curve={height=6pt}, from=3-3, to=1-3]
	\arrow["{\beta\otimes\id_1}", curve={height=-6pt}, from=1-3, to=1-1]
	\arrow["{\beta\otimes\id_2}", curve={height=-6pt}, from=3-3, to=3-1]
	\arrow["{\id_1\otimes\beta}"', curve={height=6pt}, from=3-1, to=1-1]
\end{tikzcd}\]
with relations $(a\otimes \id_y)\circ (\id_x\otimes b)=(\id_x\otimes b)\circ (a\otimes \id_y)$ for $a, b\in \{\alpha,\beta\}$, $x,y\in \{1,2\}$ such that the composition is well defined, as well as $\alpha\beta\otimes \id_x=0$ and $1_x\otimes \alpha\beta=0$ for $x\in \{1,2\}$.
The algebra $A\otimes A$ is equipped with the induced partial order ${1\otimes 1< 1\otimes 2, 2\otimes 1}$ and $2\otimes 2>1\otimes 2, 2\otimes 1$
The projectives are given by 
% https://q.uiver.app/#q=WzAsOSxbMiwwLCIoMSwxKSJdLFsxLDEsIigxLDIpIl0sWzMsMSwiKDIsMSkiXSxbMiwyLCIoMiwyKSJdLFswLDIsIigxLDEpIl0sWzQsMiwiKDEsMSkiXSxbMywzLCIoMSwyKSJdLFsxLDMsIigyLDEpIl0sWzIsNCwiKDEsMSkiXSxbMCwxXSxbMCwyXSxbMSwzXSxbMiwzXSxbMSw0XSxbMiw1XSxbNSw2XSxbNCw3XSxbMyw3XSxbMyw2XSxbNiw4XSxbNyw4XV0=
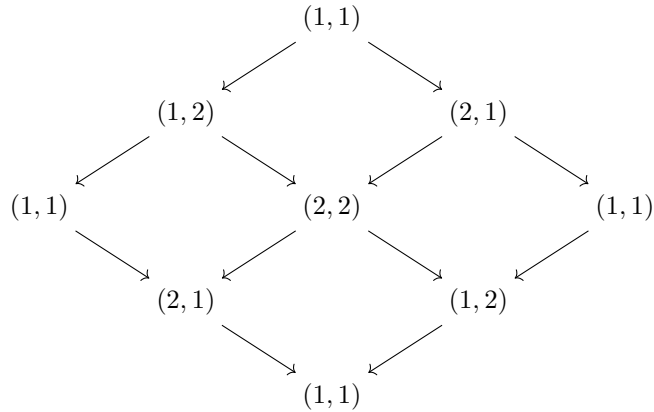
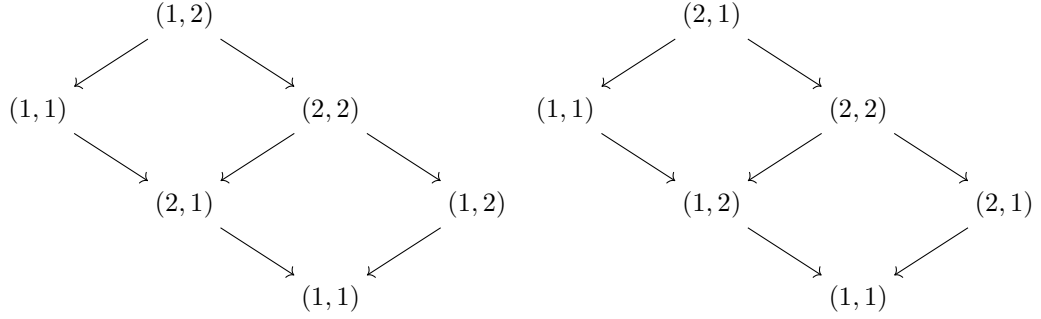
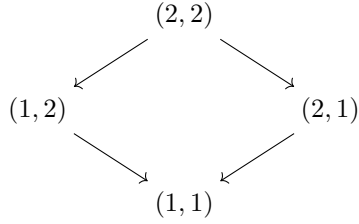
\[\begin{tikzcd}[ampersand replacement=\&]
	\&\& {(1,1)} \\
	\& {(1,2)} \&\& {(2,1)} \\
	{(1,1)} \&\& {(2,2)} \&\& {(1,1)} \\
	\& {(2,1)} \&\& {(1,2)} \\
	\&\& {(1,1)}
	\arrow[from=1-3, to=2-2]
	\arrow[from=1-3, to=2-4]
	\arrow[from=2-2, to=3-3]
	\arrow[from=2-4, to=3-3]
	\arrow[from=2-2, to=3-1]
	\arrow[from=2-4, to=3-5]
	\arrow[from=3-5, to=4-4]
	\arrow[from=3-1, to=4-2]
	\arrow[from=3-3, to=4-2]
	\arrow[from=3-3, to=4-4]
	\arrow[from=4-4, to=5-3]
	\arrow[from=4-2, to=5-3]
\end{tikzcd}\]
% https://q.uiver.app/#q=WzAsNixbMSwwLCIoMSwyKSJdLFswLDEsIigxLDEpIl0sWzIsMSwiKDIsMikiXSxbMSwyLCIoMiwxKSJdLFszLDIsIigxLDIpIl0sWzIsMywiKDEsMSkiXSxbMiwzXSxbMCwyXSxbMiw0XSxbNCw1XSxbMSwzXSxbMyw1XSxbMCwxXV0=
$\begin{tikzcd}[ampersand replacement=\&]
	\& {(1,2)} \\
	{(1,1)} \&\& {(2,2)} \\
	\& {(2,1)} \&\& {(1,2)} \\
	\&\& {(1,1)}
	\arrow[from=2-3, to=3-2]
	\arrow[from=1-2, to=2-3]
	\arrow[from=2-3, to=3-4]
	\arrow[from=3-4, to=4-3]
	\arrow[from=2-1, to=3-2]
	\arrow[from=3-2, to=4-3]
	\arrow[from=1-2, to=2-1]
\end{tikzcd}$
% https://q.uiver.app/#q=WzAsNixbMSwwLCIoMiwxKSJdLFswLDEsIigxLDEpIl0sWzIsMSwiKDIsMikiXSxbMSwyLCIoMSwyKSJdLFszLDIsIigyLDEpIl0sWzIsMywiKDEsMSkiXSxbMiwzXSxbMCwyXSxbMiw0XSxbNCw1XSxbMSwzXSxbMyw1XSxbMCwxXV0=
$\begin{tikzcd}[ampersand replacement=\&]
	\& {(2,1)} \\
	{(1,1)} \&\& {(2,2)} \\
	\& {(1,2)} \&\& {(2,1)} \\
	\&\& {(1,1)}
	\arrow[from=2-3, to=3-2]
	\arrow[from=1-2, to=2-3]
	\arrow[from=2-3, to=3-4]
	\arrow[from=3-4, to=4-3]
	\arrow[from=2-1, to=3-2]
	\arrow[from=3-2, to=4-3]
	\arrow[from=1-2, to=2-1]
\end{tikzcd}$
and
% https://q.uiver.app/#q=WzAsNCxbMSwwLCIoMiwyKSJdLFswLDEsIigxLDIpIl0sWzIsMSwiKDIsMSkiXSxbMSwyLCIoMSwxKSJdLFswLDFdLFswLDJdLFsyLDNdLFsxLDNdXQ==
\[\begin{tikzcd}[ampersand replacement=\&]
	\& {(2,2)} \\
	{(1,2)} \&\& {(2,1)} \\
	\& {(1,1)}
	\arrow[from=1-2, to=2-1]
	\arrow[from=1-2, to=2-3]
	\arrow[from=2-3, to=3-2]
	\arrow[from=2-1, to=3-2]
\end{tikzcd}\]
and the standard modules are given by $\Delta_{(1,1)}=L_{(1,1)}$, $\Delta_{(1,2)}$ is the module with Loewy diagram
% https://q.uiver.app/#q=WzAsMixbMSwwLCIoMSwyKSJdLFswLDEsIigxLDEpIl0sWzAsMV1d
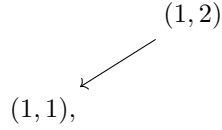
\[\begin{tikzcd}[ampersand replacement=\&]
	\& {(1,2)} \\
	{(1,1),}
	\arrow[from=1-2, to=2-1]
\end{tikzcd}\]
 $\Delta_{(2,1)}$ is the module with Loewy diagram
% https://q.uiver.app/#q=WzAsMixbMSwwLCIoMiwxKSJdLFswLDEsIigxLDEpIl0sWzAsMV1d
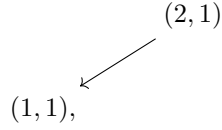
\[\begin{tikzcd}[ampersand replacement=\&]
	\& {(2,1)} \\
	{(1,1),}
	\arrow[from=1-2, to=2-1]
\end{tikzcd}\]
and $\Delta_{(2,2)}=P_{(2,2)}$.
Moreover, the costandard modules are given by $L_{(1,1)}$,\\
% https://q.uiver.app/#q=WzAsMixbMSwxLCIoMSwyKSJdLFswLDAsIigxLDEpIl0sWzEsMF1d
$
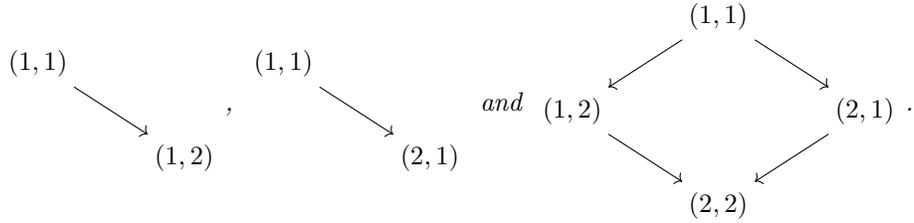
\begin{tikzcd}[ampersand replacement=\&]
	{(1,1)} \\
	\& {(1,2)}
	\arrow[from=1-1, to=2-2]
\end{tikzcd}$,
% https://q.uiver.app/#q=WzAsMixbMSwxLCIoMiwxKSJdLFswLDAsIigxLDEpIl0sWzEsMF1d
$\begin{tikzcd}[ampersand replacement=\&]
	{(1,1)} \\
	\& {(2,1)}
	\arrow[from=1-1, to=2-2]
\end{tikzcd}$
and
% https://q.uiver.app/#q=WzAsNCxbMiwxLCIoMiwxKSJdLFsxLDAsIigxLDEpIl0sWzAsMSwiKDEsMikiXSxbMSwyLCIoMiwyKSJdLFsxLDBdLFsxLDJdLFsyLDNdLFswLDNdXQ==
$\begin{tikzcd}[ampersand replacement=\&]
	\& {(1,1)} \\
	{(1,2)} \&\& {(2,1)} \\
	\& {(2,2)}
	\arrow[from=1-2, to=2-3]
	\arrow[from=1-2, to=2-1]
	\arrow[from=2-1, to=3-2]
	\arrow[from=2-3, to=3-2]
\end{tikzcd}$.\\
Now by Theorem \ref{thm_tensor}, $A\otimes A$ also has an exact Borel subalgebra $B\otimes B$ given by 
% https://q.uiver.app/#q=WzAsNCxbMCwwLCIoMSwxKSJdLFsxLDAsIigyLDEpIl0sWzAsMSwiKDEsMikiXSxbMSwxLCIoMiwyKSJdLFswLDIsIlxcaWRfMVxcb3RpbWVzXFxhbHBoYSIsMl0sWzIsMywiXFxhbHBoYVxcb3RpbWVzXFxpZF8yIiwyXSxbMCwxLCJcXGFscGhhXFxvdGltZXMgXFxpZF8xIl0sWzEsMywiXFxpZF8yXFxvdGltZXNcXGFscGhhIl1d
\[\begin{tikzcd}[ampersand replacement=\&]
	{(1,1)} \& {(2,1)} \\
	{(1,2)} \& {(2,2)}
	\arrow["{\id_1\otimes\alpha}"', from=1-1, to=2-1]
	\arrow["{\alpha\otimes\id_2}"', from=2-1, to=2-2]
	\arrow["{\alpha\otimes \id_1}", from=1-1, to=1-2]
	\arrow["{\id_2\otimes\alpha}", from=1-2, to=2-2]
\end{tikzcd}\]
with relation $(\alpha\otimes \id_2)\circ (\id_1\otimes \alpha)=(\id_2\otimes \alpha)\circ (\alpha\otimes \id_1)$. However, this is not regular, since $\rad(\Delta^{A\otimes A}_{(2,2)})$ has no costandard filtration, so that by \cite[Theorem D]{Conde}, $A\otimes A$ has no regular exact Borel subalgebra. To make this more explicit, let us calculate the up to isomorphism unique Morita equivalent quasi-hereditary algebra admitting a basic regular exact Borel subalgebra, as well as that subalgebra. Following the construction in \cite{KKO}, we calculate the Yoneda algebra $\Ext^*_A(\Delta, \Delta)$, together with a collection of higher multiplications $m_n:\Ext^*_A(\Delta, \Delta)^{\otimes_L n}\rightarrow \Ext^*_A(\Delta, \Delta)$, which specify the A-infinity algebra structure on $\Ext^*_A(\Delta, \Delta)$. For an introduction to A-infinity algebras, see for example \cite{Keller}. We begin by fixing projective resolutions
% https://q.uiver.app/#q=WzAsNCxbMCwwLCJQX3syLDJ9Il0sWzEsMCwiUF97MSwyfVxcb3BsdXMgUF97MiwxfSJdLFszLDAsIlBfezEsMX0iXSxbNCwwLCJcXERlbHRhX3sxLDF9Il0sWzAsMSwiXFxiZWdpbntwbWF0cml4fXJfe1xcYmV0YX1cXG90aW1lc1xcaWRfMlxcXFxcXC1cXGlkXzJcXG90aW1lcyByX3tcXGJldGF9XFxlbmR7cG1hdHJpeH0iXSxbMSwyLCJcXGJlZ2lue3BtYXRyaXh9XFxpZF8xXFxvdGltZXMgcl97XFxiZXRhfSAmIHJfe1xcYmV0YX1cXG90aW1lcyBcXGlkXzFcXGVuZHtwbWF0cml4fSJdLFsyLDNdXQ==
\[\begin{tikzcd}[ampersand replacement=\&, column sep= large]
	{P_{2,2}} \& {P_{1,2}\oplus P_{2,1}} \&\& {P_{1,1}} \& {\Delta_{1,1}}
	\arrow["\begin{array}{c} \begin{pmatrix}r_{\beta}\otimes\id_2\\\-\id_2\otimes r_{\beta}\end{pmatrix} \end{array}", from=1-1, to=1-2]
	\arrow["{\begin{pmatrix}\id_1\otimes r_{\beta} & r_{\beta}\otimes \id_1\end{pmatrix}}", from=1-2, to=1-4]
	\arrow[from=1-4, to=1-5]
\end{tikzcd}\]
% https://q.uiver.app/#q=WzAsMyxbMCwwLCJQX3syLDJ9Il0sWzEsMCwiUF97MSwyfSJdLFsyLDAsIlxcRGVsdGFfezEsMn0iXSxbMCwxLCJyX3tcXGJldGF9XFxvdGltZXNcXGlkXzIiXSxbMSwyXV0=
\[\begin{tikzcd}[ampersand replacement=\&, column sep= large]
	{P_{2,2}} \& {P_{1,2}} \& {\Delta_{1,2}}
	\arrow["{r_{\beta}\otimes\id_2}", from=1-1, to=1-2]
	\arrow[from=1-2, to=1-3]
\end{tikzcd}\]
% https://q.uiver.app/#q=WzAsMyxbMCwwLCJQX3syLDJ9Il0sWzEsMCwiUF97MSwyfSJdLFsyLDAsIlxcRGVsdGFfezIsMX0iXSxbMCwxLCJcXGlkXzJcXG90aW1lcyByX3tcXGJldGF9Il0sWzEsMl1d
\[\begin{tikzcd}[ampersand replacement=\&, column sep= large]
	{P_{2,2}} \& {P_{1,2}} \& {\Delta_{2,1}}
	\arrow["{\id_2\otimes r_{\beta}}", from=1-1, to=1-2]
	\arrow[from=1-2, to=1-3]
\end{tikzcd}\]
and 
% https://q.uiver.app/#q=WzAsMixbMCwwLCJQX3syLDJ9Il0sWzEsMCwiXFxEZWx0YV97MiwxfSJdLFswLDFdXQ==
\[\begin{tikzcd}[ampersand replacement=\&, column sep= large]
	{P_{2,2}} \& {\Delta_{2,1}}
	\arrow[from=1-1, to=1-2]
\end{tikzcd}\]
of the standard modules. 
Then, we can explicitly calculate the non-identity, non-zero homomorphism spaces 
\begin{align*}
    &\Hom_A(\Delta_{1,1}, \Delta_{1,2})=\field \id_1\otimes r_{\alpha}, 
    &\Hom_A(\Delta_{1,1}, \Delta_{2,1})=\field r_{\alpha}\otimes \id_1,\\
    &\Hom_A(\Delta_{1,2}, \Delta_{2,2})=\field r_{\alpha}\otimes \id_2, 
     &\Hom_A(\Delta_{1,2}, \Delta_{2,2})=\field \id_2\otimes r_{\alpha}, \\ 
    &\Hom_A(\Delta_{1,1}, \Delta_{2,2})=\field r_{\alpha}\otimes r_{\alpha}&
\end{align*}    
as well as the non-zero extensions between the standard modules
\begin{align*}
    &\Ext^1_A(\Delta_{1,1}, \Delta_{1,2})=\field [\phi], 
    &\Ext^1_A(\Delta_{1,1}, \Delta_{2,1})=\field [\phi'], 
    &\Ext^1_A(\Delta_{1,2}, \Delta_{2,2})=\field [\eta], \\ 
     &\Ext^1_A(\Delta_{2,1}, \Delta_{2,2})=\field [\eta'], 
    &\Ext^1_A(\Delta_{1,1}, \Delta_{2,2})=\field [\nu]\oplus \field [\nu'],
    &\Ext^2_A(\Delta_{1,1}, \Delta_{1,2})=\field [\varphi],\\
\end{align*}
where 
\begin{align*}
    &\phi_0=\pi_{P_{1,2}}, \phi_1=\id_{P_{2,2}}, \phi_n=0 &\forall n\geq 3\\
    &\phi'_0=\pi_{P_{2,1}}, \phi_1=-\id_{P_{2,2}}, \phi'_n=0 &\forall n\geq 3\\
    &\eta_0=\id_{P_{2,2}}, \eta_n=0 &\forall n\geq 2\\
    &\eta_0'=\id_{P_{2,2}}, \eta'_n=0 &\forall n\geq 2\\
    &\nu_0=(r_{\alpha}\otimes\id_2)\circ \pi_{P_{1,2}},\nu_2=0 &\forall n\geq 2\\
    &\nu'_0=(\id_2\otimes r_{\alpha})\circ \pi_{P_{2,1}}, \nu'_2=0 &\forall n\geq 2\\
    &\varphi_0=\id_{P_{2,2}}, \varphi=0 &\forall n\geq 2.\\
\end{align*}
The non-zero compositions among these are given by
\begin{align*}
  \field r_{\alpha}\otimes \id_2 \circ  (\id_1\otimes r_{\alpha})&= (\id_2\otimes r_{\alpha})\circ \field r_{\alpha}\otimes \id_1=\field r_{\alpha}\otimes r_{\alpha}\\
  [\eta']\circ (r_{\alpha}\otimes \id_1)&=(r_{\alpha}\otimes \id_2)\circ [\phi]=[\nu]\\
  [\eta] \circ(\id_1\otimes r_{\alpha})&=(\id_2\otimes r_{\alpha})\circ [\phi']=[\nu']\\
  [\eta]\circ [\psi]&=-[\psi]\circ [\eta]=[\varphi].
\end{align*} 
Note that since the standard modules form a directed poset of height three, it follows that $\Ext^*_A(\Delta, \Delta)^{\otimes_L n}=(0)$ for all $n\geq 3$, so that $m_n=0$ for all $n\geq 3$. Hence, the only non-zero multiplication is $m_2$, which comes from the compositions above.
We can thus calculate, using the algorithm and notation from \cite{KKO}, that the corresponding bocs is given by 
% https://q.uiver.app/#q=WzAsNCxbMCwwLCIoMSwxKSJdLFsyLDAsIigxLDIpIl0sWzAsMiwiKDIsMSkiXSxbMiwyLCIoMiwyKSJdLFsyLDMsIltcXGV0YSddXioiLDEseyJjdXJ2ZSI6LTF9XSxbMCwxLCIoXFxpZF8xXFxvdGltZXMgcl97XFxhbHBoYX0pXioiLDAseyJjdXJ2ZSI6LTEsInN0eWxlIjp7ImJvZHkiOnsibmFtZSI6ImRhc2hlZCJ9fX1dLFswLDMsIltcXG51XV4qIiwwLHsiY3VydmUiOi0yfV0sWzEsMywiW1xcZXRhXV4qIiwxLHsiY3VydmUiOjF9XSxbMCwzLCJbXFxudSddXioiLDIseyJjdXJ2ZSI6Mn1dLFswLDEsIltcXHBoaV1eKiIsMSx7ImN1cnZlIjoxfV0sWzAsMiwiW1xccGhpJ11eKiIsMSx7ImN1cnZlIjotMX1dLFswLDIsIihyX3tcXGFscGhhfVxcb3RpbWVzIFxcaWRfMSleKiIsMix7ImN1cnZlIjoxLCJzdHlsZSI6eyJib2R5Ijp7Im5hbWUiOiJkYXNoZWQifX19XSxbMCwzLCIocl97XFxhbHBoYX1cXG90aW1lcyByX3tcXGFscGhhfSleKiIsMSx7InN0eWxlIjp7ImJvZHkiOnsibmFtZSI6ImRhc2hlZCJ9fX1dLFsxLDMsIihyX3tcXGFscGhhfVxcb3RpbWVzIFxcaWRfMileKiIsMCx7ImN1cnZlIjotMSwic3R5bGUiOnsiYm9keSI6eyJuYW1lIjoiZGFzaGVkIn19fV0sWzIsMywiKFxcaWRfMlxcb3RpbWVzIHJfe1xcYWxwaGF9KV4qIiwyLHsiY3VydmUiOjEsInN0eWxlIjp7ImJvZHkiOnsibmFtZSI6ImRhc2hlZCJ9fX1dXQ==
\[\begin{tikzcd}[ampersand replacement=\&]
	{(1,1)} \&\& {(1,2)} \\
	\\
	{(2,1)} \&\& {(2,2)}
	\arrow["{(\id_1\otimes r_{\alpha})^*}", curve={height=-6pt}, dashed, from=1-1, to=1-3]
	\arrow["{[\phi]^*}"{description}, curve={height=6pt}, from=1-1, to=1-3]
	\arrow["{[\phi']^*}"{description}, curve={height=-6pt}, from=1-1, to=3-1]
	\arrow["{(r_{\alpha}\otimes \id_1)^*}"', curve={height=6pt}, dashed, from=1-1, to=3-1]
	\arrow["{[\nu]^*}", curve={height=-12pt}, from=1-1, to=3-3]
	\arrow["{[\nu']^*}"', curve={height=12pt}, from=1-1, to=3-3]
	\arrow["{(r_{\alpha}\otimes r_{\alpha})^*}"{description}, dashed, from=1-1, to=3-3]
	\arrow["{[\eta]^*}"{description}, curve={height=6pt}, from=1-3, to=3-3]
	\arrow["{(r_{\alpha}\otimes \id_2)^*}", curve={height=-6pt}, dashed, from=1-3, to=3-3]
	\arrow["{[\eta']^*}"{description}, curve={height=-6pt}, from=3-1, to=3-3]
	\arrow["{(\id_2\otimes r_{\alpha})^*}"', curve={height=6pt}, dashed, from=3-1, to=3-3]
\end{tikzcd}\]
with relation $[\eta']^*\cdot [\phi']^*= [\eta]^*\cdot [\phi]^*$ and differential 
\begin{align*}
    \partial [\nu]^*= [\eta']^*\otimes (r_{\alpha}\otimes \id_1)^*+ (r_{\alpha}\otimes \id_2)^*\otimes [\phi]^*\\
     \partial [\nu']^*= [\eta]^*\otimes (\id_1\otimes r_{\alpha})^*+ (\id_2\otimes r_{\alpha})^*\otimes [\phi']^*\\
     \partial (r_{\alpha}\otimes r_{\alpha})^*=(\id_2\otimes r_{\alpha})^*\otimes (r_{\alpha}\otimes \id_1)^*+(r_{\alpha}\otimes \id_2)^* \otimes (\id_1\otimes r_{\alpha})^* 
\end{align*}
Let us denote by $Q_0$ the quiver consisting of the solid arrows above. Then $B:=\field Q_0/\langle [\eta']^*\cdot [\phi']^*-[\eta]^*\cdot [\phi]^*\rangle$ is the up to isomorphism unique basic algebra that is a regular exact Borel subalgebra of an algebra $R$ Morita equivalent to $A$. In fact, by \cite[Theorem A]{Conde}, $R$ can be chosen as $\End_A(A\oplus P_{2,2}'\oplus P_{2,2}'')^{\op}$, where $P_{2,2}= P_{2,2}'=P_{2,2}''$. We claim that the embedding $\iota:  B\rightarrow R$ given by
\begin{align*}
   &e_{1,1}\mapsto (\id_{P_1}+\id_{P_{2,2}'}+\id_{P_{2,2}''})
 &e_{1,2}\mapsto \id_{P_{1,2}}\\
  &e_{2,1}\mapsto \id_{P_{2,1}}
 & e_{2,2}\mapsto \id_{P_{2,2}}\\
  &[\phi]^* \mapsto \begin{pmatrix}\id_1\otimes \alpha \\ -\beta\otimes\id_2\end{pmatrix}: P_{1,2}\rightarrow P_{1,1}\oplus P_{2,2}'
  &[\nu]^* \mapsto (\id: P_{2,2}\rightarrow P_{2,2}')\\
 &[\phi']^* \mapsto \begin{pmatrix}\alpha\otimes \id_1 \\ -\id_2\otimes \beta\end{pmatrix}:P_{2,1}\rightarrow P_{1,1}\oplus P_{2,2}''
 &[\nu']^* \mapsto (\id: P_{2,2}\rightarrow P_{2,2}'')\\
 &[\eta]^* \mapsto (\alpha\otimes \id_2:P_{2,2}\rightarrow P_{1,2})
 &[\eta']^* \mapsto (\id_2\otimes \alpha: P_{2,2}\rightarrow P_{2,1})
\end{align*}
turns $B$ into a regular exact Borel subalgebra of $R$. It is clear that $B$ is via $\iota$ a directed subalgebra of $A$. Hence, in order to show that $B$ is an exact Borel subalgebra of $A$, it suffices by \cite[Theorem A]{Koenig} to show that the costandard modules of $R$ restrict to the injective $B$-modules. Since $\nabla_{1,1}^A$ and thus $\nabla_{1,1}^R$  is simple, it automatically restricts to $I_{1,1}^B=L_{1,1}^B$. Moreover, $\nabla_{1,2}^A$ and thus $\nabla_{1,2}^R$ have exactly two composition factors, as has $I_{1,2}^B$. Therefore, it suffices to check that the restriction of $\nabla_{1,2}^R$ to a left $B$-module is not semisimple. Let $f:P_{1,1}\rightarrow \nabla_{1,2}^A$ be the canonical projection. Then $\iota([\phi]^*)\cdot f\neq 0$. Viewing $f$ as an element of $\Hom_A(A\oplus  P_{2,2}'\oplus P_{2,2}'', \nabla_{1,2}^A)$, we can see that this implies that $\nabla_{1,2}^R$ is not semisimple as a $B$-module. By symmetry, the same holds for $\nabla_{2,1}^R$. It remains to be shown that ${}_{B|}\nabla_{2,2}^R$ is injective. It is easy to check that $\nabla_{2,2}^R\cong \Hom_A(A\oplus  P_{2,2}'\oplus P_{2,2}'', \nabla_{2,2}^A)$ is six-dimensional, as is $I_{2,2}^B$, so that it suffices to show that  ${}_{B|}\nabla_{2,2}^R$ has a simple socle. More precisely,  $\Hom_A(A\oplus  P_{2,2}'\oplus P_{2,2}'', \nabla_{2,2}^A)$ is spanned by the homomorphisms
\begin{align*}
    f_1:P_{1,1}&\rightarrow \nabla_{2,2}^A, &e_{1,1}\mapsto (\alpha\otimes\alpha)^*\\
    f_2:P_{2,1}&\rightarrow \nabla_{2,2}^A, &e_{2,1}\mapsto (\id_2\otimes\alpha)^*\\
    f_3:P_{1,2}&\rightarrow \nabla_{2,2}^A, &e_{1,2}\mapsto (\alpha\otimes\id_2)^*,\\
    f_4:P_{2,2}&\rightarrow \nabla_{2,2}^A, &e_{2,2}\mapsto e_{2,2}^*\\
    f_4':P_{2,2}'&\rightarrow \nabla_{2,2}^A, &e_{2,2}\mapsto e_{2,2}^* \\
    f_4'':P_{2,2}'&\rightarrow \nabla_{2,2}^A, &e_{2,2}\mapsto e_{2,2}^*,
\end{align*}
where we have identified the costandard module $\nabla_{2,2}^A$ with the corresponding submodule of $\Hom_{\field}(P_{2,2}^{A^{\op}}, \field)\cong I_{2,2}$.
Moreover,
\begin{align*}
    \iota([\nu]^*)\cdot f_4'= \iota([\nu']^*)\cdot f_4''&=f_4\\
    \iota([\eta]^*)\cdot f_3= \iota([\eta']^*)\cdot f_2&=f_4\\
     \iota([\eta]^*)\cdot \iota([\phi]^*)\cdot f_1&=f_4.
\end{align*}
Therefore, the socle of ${}_{B|}\nabla_{2,2}^R$ is $\field f_4\cong L_{2,2}^B$.\\
Let us consider the one-extension between the simple $B$-modules $L_{1,1}^B$ and $L_{2,2}^B$ given by $Be_{1,1}^B/B\{[\phi']^*, [\phi]^*, [\nu']^*\}$. 
Inducing this to an $R$-module, we obtain
\begin{align*}
    &R\otimes_B Be_{1,1}^B/B\{[\phi']^*, [\phi]^*, [\nu']^*\}\cong R\iota (e_{1,1})/(R\{\iota([\phi']^*), \iota([\phi]^*), \iota([\nu']^*)\})\\
    \cong &R(\id_{P_{1,1}}+\id{P_{2,2}'})/(R\{\iota([\phi']^*), \alpha\otimes \id_1\})\\
    \cong &\Hom_A(A\oplus P_{2,2}'\oplus P_{2,2}'', (P_{1,1}\oplus P_{2,2})/(\im(\alpha\otimes\id_1)+\im(\id\otimes_1\alpha-\beta\otimes\id_2))).
\end{align*}
Note that $(P_{1,1}\oplus P_{2,2})/(\im(\alpha\otimes\id_1)+\im(\id\otimes_1\alpha-\beta\otimes\id_2)))$ has Loewy diagram % https://q.uiver.app/#q=WzAsNSxbMywwLCIoMSwxKSJdLFsyLDEsIigxLDIpIl0sWzEsMCwiKDIsMikiXSxbMSwyLCIoMSwxKSJdLFswLDEsIigyLDEpIl0sWzAsMV0sWzEsM10sWzIsMV0sWzIsNF0sWzQsM11d
\[\begin{tikzcd}[ampersand replacement=\&]
	\& {(2,2)} \&\& {(1,1)} \\
	{(2,1)} \&\& {(1,2)} \\
	\& {(1,1)}
	\arrow[from=1-2, to=2-1]
	\arrow[from=1-2, to=2-3]
	\arrow[from=1-4, to=2-3]
	\arrow[from=2-1, to=3-2]
	\arrow[from=2-3, to=3-2]
\end{tikzcd}\]
as an $A$-module, so that $R\otimes_B Be_{1,1}^B/B\{[\phi']^*, [\phi]^*, [\nu']^*\}$ has the same Loewy diagram as an $R$-module.
Similarly, $Be_{1,1}^B/B\{[\phi']^*, [\phi]^*, [\nu]^*\}$ induces to an $R$-module with Loewy diagram
% https://q.uiver.app/#q=WzAsNSxbMCwwLCIoMSwxKSJdLFszLDEsIigxLDIpIl0sWzIsMCwiKDIsMikiXSxbMiwyLCIoMSwxKSJdLFsxLDEsIigyLDEpIl0sWzEsM10sWzIsMV0sWzIsNF0sWzQsM10sWzAsNF1d
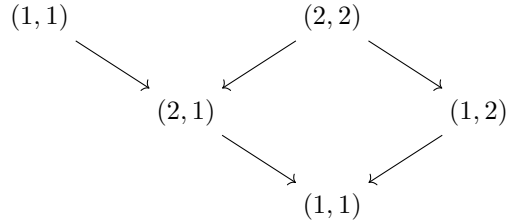
\[\begin{tikzcd}[ampersand replacement=\&]
	{(1,1)} \&\& {(2,2)} \\
	\& {(2,1)} \&\& {(1,2)} \\
	\&\& {(1,1)}
	\arrow[from=1-1, to=2-2]
	\arrow[from=1-3, to=2-2]
	\arrow[from=1-3, to=2-4]
	\arrow[from=2-2, to=3-3]
	\arrow[from=2-4, to=3-3]
\end{tikzcd}\]
In particular, $R\otimes_B -$ induces an isomorphism $$\Ext^1_B(L_{1,1}^B, L_{2,2}^B)\rightarrow \Ext^1_R(\Delta_{1,1}^R, \Delta_{2,2}^R).$$
Let us consider the extension $Be_{1,1}/B\{[\nu']^*, [\nu]^*, [\phi']^*\}$ of the simple $B$-modules $L_{1,1}^B$ and $L_{1,2}^B$. Inducing this to an $R$-module, we obtain
\begin{align*}
    R\otimes_B Be_{1,1}/B\{[\nu']^*, [\nu]^*, [\phi']^*\}&\cong R\id_{1,1}/R(\alpha\otimes\id_2)\\
    &\cong \Hom(A\oplus P_{2,2}'\oplus P_{2,2}'', P_{1,1}/\im(\alpha\otimes \id_1)).
\end{align*}
The $A$-module $P_{1,1}/\im(\alpha\otimes \id_1)$ has Loewy diagram
% https://q.uiver.app/#q=WzAsMyxbMiwwLCIoMSwxKSJdLFsxLDEsIigxLDIpIl0sWzAsMiwiKDEsMSkiXSxbMCwxXSxbMSwyXV0=
\[\begin{tikzcd}[ampersand replacement=\&]
	\&\& {(1,1)} \\
	\& {(1,2)} \\
	{(1,1),}
	\arrow[from=1-3, to=2-2]
	\arrow[from=2-2, to=3-1]
\end{tikzcd}\]
which implies that $R\otimes_B-$ induces an isomorphism
$$\Ext^1_B(L_{1,1}^B, L_{1,2}^B)\rightarrow \Ext^1_R(\Delta_{1,1}^R, \Delta_{1,2}^R).$$
Analogously, $R\otimes_B-$ induces an isomorphism
$$\Ext^1_B(L_{1,1}^B, L_{2,1}^B)\rightarrow \Ext^1_R(\Delta_{1,1}^R, \Delta_{2,1}^R).$$
Similarly, one can check that $R\otimes_B-$ induces isomorphisms
$$\Ext^1_B(L_{1,2}^B, L_{2,2}^B)\rightarrow \Ext^1_R(\Delta_{1,2}^R, \Delta_{2,2}^R)$$
and
$$\Ext^1_B(L_{2,1}^B, L_{2,2}^B)\rightarrow \Ext^1_R(\Delta_{2,1}^R, \Delta_{2,2}^R).$$
Finally, since the up to a scalar unique 2-extension of $B$ given by 
% https://q.uiver.app/#q=WzAsNixbMSwwLCJMX3syLDJ9XkIiXSxbMiwwLCJCZV97MSwyfSJdLFszLDAsIkJlX3sxLDF9L0JcXHtbXFxudSddXiosIFtcXG51XV4qLCBbXFxwaGknXV4qXFx9Il0sWzQsMCwiTF97MSwxfV5CIl0sWzAsMCwiKDApIl0sWzUsMCwiKDApIl0sWzAsMV0sWzEsMl0sWzIsM10sWzQsMF0sWzMsNV1d
\[\begin{tikzcd}[ampersand replacement=\&]
	{(0)} \& {L_{2,2}^B} \& {Be_{1,2}} \& {Be_{1,1}/B\{[\nu']^*, [\nu]^*, [\phi']^*\}} \& {L_{1,1}^B} \& {(0)}
	\arrow[from=1-1, to=1-2]
	\arrow[from=1-2, to=1-3]
	\arrow[from=1-3, to=1-4]
	\arrow[from=1-4, to=1-5]
	\arrow[from=1-5, to=1-6]
\end{tikzcd}\] is a product of the two up to a non-zero scalar unique non-trivial 1-extensions corresponding in $\Ext^1_B(L_{1,1}^B, L_{1,2}^B)$ and $\Ext^1_B(L_{1,2}^B, L_{2,2}^B)$, we obtain that it maps under $R\otimes_B-$ to the product of the up to a non-zero scalar unique non-trivial 1-extensions in $\Ext^1_R(\Delta_{1,1}^R, \Delta_{1,2}^R)$ and $\Ext^1_R(\Delta_{1,2}^R, \Delta_{2,2}^R)$, which is the up to a non-zero scalar unique non-tirival 2-extension of $R$. Therefore, we can conclude that $B$ is regular.
\end{example}
This provides an example of the fact that tensor products of regular exact Borel subalgebras are in general not regular. In fact, regularity of $B\otimes B'$ holds only in rare cases:
\begin{proposition}\label{proposition_regular}
	Suppose $B$ and $B'$ are regular. Then the following statements are equivalent: 
		\begin{enumerate}
			\item $B\otimes B'$ is regular.
			\item $B\otimes B'$ is homological.
            \item One of the following statements holds:
            \begin{enumerate}
			\item $A$ and $A'$ are directed.
			\item $A^{\op}$ and $(A')^{\op}$ are directed.
			\item $A$ is semisimple.
            \item $A'$ is semisimple.
            \end{enumerate}
		\end{enumerate}
\end{proposition}
\begin{proof}
	Let $P^{\cdot}(L^B), P^{\cdot}(L^{B'}), P^{\cdot}(\Delta^{A})$ and $P^{\cdot}(\Delta^{A'})$ be projective resolutions of the modules $L^B, L^{B'}, \Delta^A$ and $\Delta^{A'}$ respectively.\\ Then $P^{\cdot}(L^B\otimes L^{B'}):=P^{\cdot}(L^B)\otimes P^{\cdot}(L^{B'})$ and $P^{\cdot}(\Delta^A\otimes \Delta^{A'}):=P^{\cdot}(\Delta^{A})\otimes P^{\cdot}(\Delta^{A'})$ are projective resolutions of $L^B\otimes L^{B'}$, respectively $\Delta^A\otimes \Delta^{A'}$ by the Künneth formula, and we obtain
  canonical isomorphisms
	\begin{align*}
		\End_B(P^{\cdot}(L^B))\otimes \End_{B'}(P^{\cdot}(L^{B'}))\cong \End_{B\otimes B'}(P^{\cdot}(L^B\otimes L^{B'}))\\
		\End_A(P^{\cdot}(\Delta^A))\otimes \End_{A'}(P^{\cdot}(\Delta^{A'}))\cong \End_{A\otimes A'}(P^{\cdot}(\Delta^A\otimes \Delta^{A'}))\\
	\end{align*}
	of dg-algebras which, again by the Künneth fromula, give rise to isomorphisms (of vector spaces)
	\begin{align*}
		 \Ext^*_{B}(L^B,L^B)\otimes \Ext^*_{B'}(L^{B'}, L^{B'})\rightarrow &\Ext^*_{B\otimes B'}(L^B\otimes L^{B'}, L^B\otimes L^{B'}),\\ &[f]\otimes [g]\mapsto [f\otimes g] \\
		 \Ext^*_{A}(\Delta^A,\Delta^A)\otimes \Ext^*_{A'}(\Delta^{A'}, \Delta^{A'})\rightarrow &\Ext^*_{A\otimes A'}(\Delta^A\otimes \Delta^{A'}, \Delta^A\otimes \Delta^{A'}),\\ &[f]\otimes [g]\mapsto [f\otimes g]. \\
	\end{align*}
	Moreover, the diagram
% https://q.uiver.app/#q=WzAsNCxbMSwwLCJcXEV4dF4qX3tCXFxvdGltZXMgQid9KExeQlxcb3RpbWVzIExee0InfSwgTF5CXFxvdGltZXMgTF57Qid9KSJdLFswLDAsIlxcRXh0Xipfe0J9KExeQixMXkIpXFxvdGltZXMgXFxFeHReKl97Qid9KExee0InfSwgTF57Qid9KSJdLFsxLDEsIlxcRXh0Xipfe0FcXG90aW1lcyBBJ30oXFxEZWx0YV5BXFxvdGltZXMgXFxEZWx0YV57QSd9LCBcXERlbHRhXkFcXG90aW1lcyBcXERlbHRhXntBJ30pIl0sWzAsMSwiXFxFeHReKl97QX0oXFxEZWx0YV5BLFxcRGVsdGFeQSlcXG90aW1lcyBcXEV4dF4qX3tBJ30oXFxEZWx0YV57QSd9LCBcXERlbHRhXntBJ30pIl0sWzAsMiwiW2hdXFxtYXBzdG8gW1xcaWRfe0FcXG90aW1lcyBBJ31cXG90aW1lcyBoXSJdLFsxLDMsIltmXVxcb3RpbWVzIFtnXVxcbWFwc3RvIFtcXGlkX0FcXG90aW1lcyBmXVxcb3RpbWVzIFtcXGlkX3tBJ31cXG90aW1lcyBnXSIsMl0sWzMsMiwiW2FdXFxvdGltZXMgW2JdXFxtYXBzdG8gW2FcXG90aW1lcyBiXSIsMl0sWzEsMCwiW2ZdXFxvdGltZXMgW2ddXFxtYXBzdG8gW2ZcXG90aW1lcyBnXSJdXQ==
\[\begin{tikzcd}[ampersand replacement=\&, column sep=huge]
	{\Ext^*_{B}(L^B,L^B)\otimes \Ext^*_{B'}(L^{B'}, L^{B'})} \& {\Ext^*_{B\otimes B'}(L^B\otimes L^{B'}, L^B\otimes L^{B'})} \\
	{\Ext^*_{A}(\Delta^A,\Delta^A)\otimes \Ext^*_{A'}(\Delta^{A'}, \Delta^{A'})} \& {\Ext^*_{A\otimes A'}(\Delta^A\otimes \Delta^{A'}, \Delta^A\otimes \Delta^{A'})}
	\arrow["{[f]\otimes [g]\mapsto [f\otimes g]}", from=1-1, to=1-2]
	\arrow["{[f]\otimes [g]\mapsto [\id_A\otimes f]\otimes [\id_{A'}\otimes g]}"', from=1-1, to=2-1]
	\arrow["{[h]\mapsto [\id_{A\otimes A'}\otimes h]}", from=1-2, to=2-2]
	\arrow["{[a]\otimes [b]\mapsto [a\otimes b]}"', from=2-1, to=2-2]
\end{tikzcd}\]
commutes. Thus the map
\begin{align*}
	\Ext^n_{B\otimes B'}(L^B\otimes L^{B'}, L^B\otimes L^{B'})\rightarrow &\Ext^n_{A\otimes A'}(\Delta^A\otimes \Delta^{A'}, \Delta^A\otimes \Delta^{A'}),\\
	 &[h]\mapsto [\id_{A\otimes A'}\otimes h]
\end{align*}
is an isomorphism respectively epimorphism if and only if the map
\begin{align*}
	\bigoplus_{i+j=n} \Ext^i_{B}(L^B,L^B)\otimes \Ext^j_{B'}(L^{B'}, L^{B'})\rightarrow &\bigoplus_{i+j=n} \Ext^i_{A}(\Delta^A,\Delta^A)\otimes \Ext^j_{A'}(\Delta^{A'}, \Delta^{A'}),\\ [f]\otimes [g]\mapsto [\id_A\otimes f]\otimes [\id_{A'}\otimes g]
\end{align*}
is an isomorphism respectively epimorphism, i.e., if for all $0\leq i\leq n$, $j:=n-i$, the map
\begin{align*}
\Ext^i_{B}(L^B,L^B)\otimes \Ext^j_{B'}(L^{B'}, L^{B'})\rightarrow &\Ext^i_{A}(\Delta^A,\Delta^A)\otimes \Ext^j_{A'}(\Delta^{A'}, \Delta^{A'}),\\
 &[f]\otimes [g]\mapsto [\id_A\otimes f]\otimes [\id_{A'}\otimes g]
\end{align*}
is an isomorphism respectively epimorphism.\\
By assumption, $B$ and $B'$ are regular, so for every $i, j>0$ the maps
\begin{align*}
	\Ext^i_{B}(L^B,L^B)\rightarrow \Ext^i_{A}(\Delta^A,\Delta^A),
	 [f]\mapsto [\id_A\otimes f]
\end{align*} 
and 
\begin{align*}
	  \Ext^j_{B'}(L^{B'}, L^{B'})\rightarrow  \Ext^j_{A'}(\Delta^{A'}, \Delta^{A'}),
		[g]\mapsto [\id_{A'}\otimes g]
	\end{align*}
	are isomorphisms. For $i=0$ the map
	\begin{align*}
		\Hom_B(L^B, L^B)\rightarrow \Hom_A(\Delta^A, \Delta^A), [f]\mapsto [\id_A\otimes f]
	\end{align*}
	corresponds to the canonical embedding
	$L\rightarrow \End_A(\Delta^A)$
	and for $j=0$ the map 
	\begin{align*}
		\Hom_{B'}(L^{B'}, L^{B'})\rightarrow \Hom_{A'}(\Delta^{A'}, \Delta^{A'}), [g]\mapsto [\id_{A'}\otimes g]
	\end{align*}
	corresponds to the canonical embedding $L'\rightarrow \End_{A'}(\Delta^{A'})$.
In particular, the map
\begin{align*}
	\Ext^i_{B}(L^B,L^B)\otimes \Ext^j_{B'}(L^{B'}, L^{B'})\rightarrow &\Ext^i_{A}(\Delta^A,\Delta^A)\otimes \Ext^j_{A'}(\Delta^{A'}, \Delta^{A'}),\\
	 &[f]\otimes [g]\mapsto [\id_A\otimes f]\otimes [\id_{A'}\otimes g]
	\end{align*}
	is an isomorphism for $i, j>0$; for $j=0$ it is an isomorphism if and only if it is an epimorphism if and only if
 $\Ext^n_A(\Delta^A, \Delta^A)=(0)$ or  $L'\cong \End_{A'}(\Delta^{A'})$; and for $i=0$ it is an isomorphism if and only if it is an epimorphism if and only if $\Ext^n_{A'}(\Delta^{A'}, \Delta^{A'})=(0)$ or
 $L\cong \End_A(\Delta^A)$.\\
 Thus $B\otimes B'$ is homological if and only if it is regular.
 Moreover, note that ${L'\cong \End_{A'}(\Delta^{A'})}$ if and only if $\Delta^{A'}\cong L'$ if and only if $A'$ is directed, and similarly 
 $L\cong \End_A(\Delta^A)$ if and only $A$ is directed.
 Thus, $B\otimes B'$ is regular if and only if 
 \begin{enumerate}
	\item $A'$ is directed or $\Ext^n_A(\Delta^A, \Delta^A)=(0)$ for all $n\geq 1$; and
	\item $A$ is directed or $\Ext^n_{A'}(\Delta^{A'}, \Delta^{A'})=(0)$ for all $n\geq 1$.
 \end{enumerate}
 Now, since the projective modules have a filtration by standard modules, $$\Ext^n_A(\Delta^A, \Delta^A)=(0)$$ for all $n\geq 1$ if and only if $\Delta^A$ is projective, and $\Ext^n_{A'}(\Delta^{A'}, \Delta^{A'})=(0)$ for all $n\geq 1$ if and only if $\Delta^{A'}$ is projective.
 Moreover, the standard modules being projective means exactly that the opposite algebra is directed.
 Thus 
 $B\otimes B'$ is regular if and only if 
 \begin{enumerate}
	\item $A'$ is directed or $A^{\op}$ is directed; and
	\item $A$ is directed or $(A')^{\op}$ is directed.
 \end{enumerate}
 This conjunction of disjunctions can be written as a disjunction of conjunctions, giving us the following cases:
 \begin{enumerate}
	\item $A'$ is directed and $A$ is directed; or
    \item $A^{\op}$ is directed and $(A')^{\op}$ is directed; or
    \item $A^{\op}$ is directed and $A$ is directed; or
    \item  $A'$ is directed and $(A')^{\op}$ is directed.
 \end{enumerate}
 Finally, note that both $A$ and $A^{\op}$ are directed if and only if $A$ is semisimple; and both $A'$ and $(A')^{\op}$ are directed if and only if $A'$ is semisimple. Thus, these correspond exactly to the cases (a)--(d) in the statement of the proposition.
\end{proof}
If we don't suppose $B$ and $B'$ to be regular a priori, we obtain the following similar statement:
\begin{corollary}
    Suppose $B$ and $B'$ are exact Borel subalgebras of $A$ and $A'$ respectively. Then the following statements are equivalent: 
		\begin{enumerate}
			\item $B\otimes B'$ is regular.
			\item $B\otimes B'$ is homological and $B$ and $B'$ are regular.
            \item One of the following statements holds:
            \begin{enumerate}
			\item $A$ and $A'$ are directed.
			\item $A^{\op}$ and $(A')^{\op}$ are directed.
			\item $A$ is semisimple and $B'$ is regular.
            \item $A'$ is semisimple and $B$ is regular.
            \end{enumerate}
		\end{enumerate}
\end{corollary}
\begin{proof}
By Proposition \ref{proposition_regular}, it follows that $(2)\Rightarrow (3)$ and $(2)\Rightarrow (1)$. Moreover, it is clear that $(3)\Rightarrow (1)$. Hence it remains to be shown that $(1)\Rightarrow (2)$, i.e. we need to show that if $B\otimes B'$ is regular, then $B$ and $B'$ are regular.
As we have seen in the proof of Proposition \ref{proposition_regular}, if $B\otimes B'$ is regular, then we have isomorphisms
    \begin{align*}
\Ext^i_{B}(L^B,L^B)\otimes \Ext^j_{B'}(L^{B'}, L^{B'})\rightarrow &\Ext^i_{A}(\Delta^A,\Delta^A)\otimes \Ext^j_{A'}(\Delta^{A'}, \Delta^{A'}),\\
 &[f]\otimes [g]\mapsto [\id_A\otimes f]\otimes [\id_{A'}\otimes g]
\end{align*}  
for all $0\leq i, j\leq n$ with $i+j\geq 1$.
In particular, we have isomorphisms
 \begin{align*}
\End_B(L^B)\otimes \Ext^j_{B'}(L^{B'}, L^{B'})\rightarrow &\End_{A}(\Delta^A)\otimes \Ext^j_{A'}(\Delta^{A'}, \Delta^{A'}),\\
 &f\otimes [g]\mapsto (\id_A\otimes f)\otimes [\id_{A'}\otimes g]
\end{align*}
and 
\begin{align*}
\Ext^i_{B}(L^B,L^B)\otimes \End_{B'}(L^{B'})\rightarrow &\Ext^i_{A}(\Delta^A,\Delta^A)\otimes \End_{A'}(\Delta^{A'}),\\
 &[f]\otimes g \mapsto [\id_A\otimes f]\otimes (\id_{A'}\otimes g)
\end{align*} 
for all $i, j\geq i$. 
Since the tensor product over $\field$ of two maps is an isomorphism if and only if both maps are isomorphisms, this implies in particular that
\begin{align*}
\Ext^j_{B'}(L^{B'}, L^{B'})\rightarrow \Ext^j_{A'}(\Delta^{A'}, \Delta^{A'}),
[g]\mapsto [\id_{A'}\otimes g],  \textup{ and }\\
\Ext^i_{B}(L^B,L^B)\rightarrow \Ext^i_{A}(\Delta^A,\Delta^A)\otimes \End_{A'}(\Delta^{A'}),
 [f]\mapsto [\id_A\otimes f].
\end{align*} 
are isomorphisms for all $i, j\geq 1$, so that $B$ and $B'$ are regular. 
\end{proof}
\section{Species of Quasi-Hereditary Algebras and their Borel subalgebras}
In \cite{triangular_matrix}, Zhu investigated under which conditions triangular matrix rings of quasi-hereditary algebras are quasi-hereditary. Inspired by this, we consider generalized species of quasi-hereditary algebras. We show that these are quasi-hereditary, and investigate how to construct exact Borel subalgebras of the species given an exact Borel subalgebra for each of the involved quasi-hereditary algebras.
\subsection{Species of Quasi-Hereditary Algebras}
Species and $\field$-species were first introduced by Gabriel in \cite{Gabriel}, and later studied in depth by, among others, Dlab and Ringel \cite{DR2, DR3, DR4, Ringel}. In their original form, categories of representations of $\field$-species describe, up to equivalence, all module categories of finite-dimensional hereditary algebras over $\field$, where $\field$ is a perfect field. This generalizes the fact that basic hereditary algebras over algebraically closed fields are isomorphic to path algebras of quivers.\\
Later, several generalisations of $\field$-species were proposed and studied, see for example \cite{Lemay, Li, julian_basic}.
In this article, we study generalized species similar to, but slightly more general than, those defined in \cite{julian_basic}.
Throughout, let $Q=(V_Q, E_Q)$ be an acyclic quiver with vertices $V_Q$ labelled by $1, \dots, n$. For a path $p$ in $Q$ we denote by $s(p)$ the start vertex and by $t(p)$ the terminal vertex. Moreover, we denote by $\max(p)$ the maximal vertex with respect to the natural order that $p$ passes through.
\begin{definition}\cite{julian_basic}
 A species on $Q$ is a collection $((A_i)_{i\in V_Q}, (M_{\alpha})_{\alpha\in E_Q})$
 where $A_i$ is a finite-dimensional $\field$-algebra for every $i \in V_{Q}$, and $M_{\alpha}$ is an $A_{t(\alpha)}$-$A_{s(\alpha)}$-bimodule.\\
 For every path $p=\alpha_k\dots\alpha_1$ in $Q$ let $$M_p:=M_{\alpha_k}\otimes_{A_{s(\alpha_k)}}M_{\alpha_{k-1}}\otimes_{A_{s(\alpha_{k-1})}} \dots \otimes_{A_{s(\alpha_2)}}M_{\alpha_1}$$
where for the trivial path $e_i$ at vertex $i$ we set $M_{e_i}:=A_i$.\\
We denote by $T((A_i)_{i\in V_Q}, (M_{\alpha})_{\alpha\in E_Q})$ the associated tensor algebra.
 \begin{align*}
    A:=\bigoplus_{p \textup{ path in }Q}M_p
\end{align*}
with multiplication
\begin{align*}
    m_p\cdot m_q:=m_p\otimes_{A_i}m_q\in M_{pq}
\end{align*}
for $p, q$ non-trivial paths in $Q$ with $s(p)=i=t(q)$, $m_p\in M_p$, $m_q\in M_q$,
and 
\begin{align*}
    a_i\cdot m_q:=a_im_q\in M_q,\\
    m_p\cdot a_i:=m_pa_i\in M_p
\end{align*}
for $p, q$ paths in $Q$ with $s(p)=i=t(q)$, $m_p\in M_p$, $m_q\in M_q$, and $a_i\in A_i$, and 
\begin{align*}
    m_p\cdot m_q:=0
\end{align*}
for $p, q$ paths in $Q$ with $s(p) \neq t(q)$.\\
Moreover, for $1\leq i\leq n$, let $A_i^A$ be the $A$-module given by
    \begin{align*}
        \bigoplus_{p} M_p\otimes A_i\rightarrow A_i\\
        m\otimes a\mapsto 0 \textup{ for }m\in M_p\neq M_{e_i}\\
        a\otimes a'\mapsto aa' \textup{ for }a\in M_{e_i}=A_i.\\
    \end{align*}
\end{definition}
Suppose  $A=T((A_i)_{1\leq i\leq n}, (M_{\alpha})_{\alpha\in E_Q})$ is the tensor algebra of a species on $Q$ and for all $1\leq i\leq n$,  $e^i_1,\dots, e^i_{m_i}$ is a set of primitive orthogonal idempotents in $A_i$. Then, for $1\leq i\leq n$ and $1\leq j\leq n_i$, we denote by $e_{ij}$ the idempotent $e^i_j\in A_i=M_{e_i}\subseteq A$.\\
The following is a basic results about species analogous to \cite[Theorem 3.3]{icpn}, see also \cite{julian_basic}:
\begin{lemma}\label{lemma_radical_species}
Let $((A_i)_{i\in V_Q}, (M_\alpha)_{\alpha\in E_Q})$ be a species on $Q$ and let $$A:=T((A_i)_{i\in V_Q}, (M_\alpha)_{\alpha\in E_Q})$$ be the associated tensor algebra. For $1\leq i\leq n$ let $e_{i1},\dots, e_{im_i}$ be a complete set of primitive orthogonal idempotents in $A_i$. 
    Then the set $\{e_{ij}|1\leq i\leq n, 1\leq j\leq m_i\}$ is a complete set of primitive orthogonal idempotents in $A$. Moreover, the radical of $A$ is given by $$\rad(A)=\bigoplus_{i=1}^n \rad(A_i)\oplus \bigoplus_{p\neq e_i\textup{ for }1\leq i\leq n}M_p.$$
    In particular, there is a bijection
    \begin{align*}
        \bigsqcup \Sim(A_i)\rightarrow \Sim(A)\\
        L\mapsto A_i^A\otimes_{A_i}L_i \textup{ for }L\in \Sim(A_i).
    \end{align*}
\end{lemma}
\begin{comment}\begin{proof}
    Clearly, these are orthogonal idempotents, and $1_A=\sum_{i=1}^n 1_{A_i}=\sum_{i=1}^n\sum_{j=1}^{m_i}e_{ij}$. It remains to be shown that they are primitive. Suppose $e_{ij}=e'+e''$ for some orthogonal idempotents $e'$ and $e''$. By definition of $A$, we can write $e'$ and $e''$ uniquely as sums of elements in $M_p$, i.e.
    \begin{align*}
        e'=\sum_{p \textup{ path in }Q}e'_p\\
        e''=\sum_{p \textup{ path in }Q}e''_p
    \end{align*}
    where $e'_p, e''_p\in M_p$. Now, by assumption, 
    \begin{align*}
      \sum_{p \textup{ path in }Q}e'_p=  e'=e_{ij}e'e_{ij}=\sum_{p \textup{ path in }Q}e_{ij}e'_pe_{ij}
    \end{align*}
    and $e_{ij}e'_pe_{ij}=0$ unless $s(p)=i=t(p)$. Comparing the coefficients yields that $e'_p=0$ unless $s(p)=i=t(p)$. Since $Q$ contains no directed cycles, this implies that $e'=e'_{e_i}\in A_i$. Similarly, we obtain $e''\in A_i$. However, $e_{ij}=e^i_j$ is a primitive idempotent in $A_i$, so that this implies that $e'=0$ or $e''=0$. Hence $e_{ij}$ is a primitive idempotent in $A$.
\end{proof}
\begin{proof}
    Clearly, every element in $R:=\bigoplus_{i=1}^n \rad(A_i)\oplus \bigoplus_{p\neq e_i\textup{ for }1\leq i\leq n}M_p$ is nilpotent, so that $R\subseteq \rad(A)$. Moreover, we have $A\cong \bigoplus_{i=1}^n\bigoplus_{j=1}^{m_i} \field e_{ij}\oplus R$. On the other hand, since $\{e_{ij}|1\leq i\leq n, 1\leq j\leq m_i\}$ is a complete set of primitive orthogonal idempotents, $A\cong \bigoplus_{i=1}^n\bigoplus_{j=1}^{m_i} \field e_{ij}\oplus \rad(A)$, so that $\rad(A)=R$.
\end{proof}
\end{comment}
From now on, suppose $(A_i, \leq_i)$ is quasi-hereditary for every $1\leq i\leq n$, where $\leq_i$ is a partial order on the isomorphism classes of simple $A_i$-modules. Let us denote by $\lesssim_i$  the pre-order on $\{e_{1}^i, \dots e_{m_i}^i\}$ corresponding to $\leq_i$, and for $1\leq j\leq m_i$ let us write
\begin{align*}
    j\lesssim_i k\Leftrightarrow e_{j}^i\lesssim_i e_{k}^i.
\end{align*}
Suppose $((A_i)_{i\in V_Q}, (M_\alpha)_{\alpha\in E_Q})$ is a species on $Q$ and denote by $$A:=T((A_i)_{i\in V_Q}, (M_\alpha)_{\alpha\in E_Q})$$ the associated tensor algebra.\\
We define a partial order $\leq_A$ on $\Sim(A)$ via
\begin{align*}
    A_i^A\otimes_{A_i} L \leq A_j^A\otimes_{A_j} L' \Leftrightarrow ((i<i') \textup{ or }(i=i' \textup{ and }L\leq L')).
\end{align*}
Then the corresponding pre-order $\lesssim_A$ on the set $\{e_{ij}|1\leq i\leq n, 1\leq j\leq m_i\}$, which is a complete set of primitive orthogonal idempotents for $A$, is given by
\begin{align*}
    e_{ij}\lesssim_A e_{kl}\Leftrightarrow (i<k) \textup{ or} (i=k\textup{ and }j\lesssim_i l).
\end{align*}
In the next lemma, we describe the standard modules of $A$ with respect to this partial order.
\begin{lemma}\label{lemma_std}
     The standard modules of $A$ are given by 
     \begin{align*}
    \Delta_{ij}\cong \bigoplus_{p \notin \field Q\sum_{i'>i}e_{i'}\field Q, s(p)=i}M_p\otimes_{A_i} \Delta_j^{A_i}
     \end{align*}
     where the multiplication is given by zero for $m_{\alpha}\in M_{\alpha}$ with $t(\alpha)>i$, and the usual multiplication in $A$ otherwise.
 \end{lemma}
 \begin{proof}
 First, note that
 \begin{align*}
     A1_{A_i}/A\sum_{i'>i}1_{A_{i'}}A1_{A_i}\cong \bigoplus_{p \notin \field Q\sum_{i'>i}e_{i'}\field Q, s(p)=i}M_p.
 \end{align*}
  where the multiplication is given by zero for $m_{\alpha}\in M_{\alpha}$ with $t(\alpha)>i$.
 By definition
     \begin{align*}
    \Delta_{ij}=Ae_{ij}/(A\sum_{(k,l)>(i,j)}e_{kl}Ae_{ij}).
     \end{align*}
     Moreover, since $Q$ has no oriented cycles, $1_{A_i}M_p=0$ for any path $p$ in $Q$ with $s(p)=i$, so that, in particular, $e_{ij'}M_p=0$ for $j<j'\leq j$. Therefore,
      \begin{align*}
    \Delta_{ij}&\cong (\bigoplus_{p \notin \field Q\sum_{i'>i}e_{i'}\field Q, s(p)=i}M_p)e_{ij}/(\bigoplus_{p \notin \field Q\sum_{i'>i}e_{i'}\field Q, s(p)=i}M_p)\sum_{j'>j}e_{ij'}A_ie_{ij}\\
    &\cong \bigoplus_{p \notin \field Q\sum_{i'>i}e_{i'}\field Q, s(p)=i}M_pe_{ij}/(M_p\sum_{j'>j}e_{ij'}A_ie_{ij})\\
    &\cong  \bigoplus_{p \notin \field Q\sum_{i'>i}e_{i'}\field Q, s(p)=i}M_p\otimes_{A_i}\Delta_j^{A_i},
     \end{align*}
     where the last isomorphism is given componentwise by
     \begin{align*}
         M_pe_{ij}/(M_p\sum_{j'>j}e_{ij'}A_ie_{ij})&\rightarrow M_p\otimes_{A_i}\Delta_j^{A_i}, \\me_{ij}+M_p\sum_{j'>j}e_{ij'}A_ie_{ij}&\mapsto m\otimes (e_{ij}+A_i\sum_{j'>j}e_{ij'}A_ie_{ij})
     \end{align*}
     with inverse
      \begin{align*}
        M_p\otimes_{A_i}\Delta_j^{A_i}&\rightarrow  M_pe_{ij}/(M_p\sum_{j'>j}e_{ij'}A_ie_{ij}),\\ m\otimes (ae_{ij}+A_i\sum_{j'>j}e_{ij'}A_ie_{ij})&\mapsto mae_{ij}+M_p\sum_{j'>j}e_{ij'}A_ie_{ij}.\qedhere
     \end{align*}
 \end{proof}
 \begin{example}\label{example_std}
 \begin{enumerate}
     \item Suppose $Q$ is directed, meaning that whenever there is an arrow $\alpha: i\rightarrow j$, then $i<j$. Then, the standard modules of $A$ are just the standard modules of $A_i$ for every $1\leq i\leq n$, viewed as $A$-modules: 
     \begin{align*}
    \Delta_{ij}&\cong \bigoplus_{p \notin \field Q\sum_{i'>i}e_{i'}\field Q, s(p)=i}M_p\otimes_{A_i}\Delta_j^{A_i}\\
    &=A_i^A\otimes_{A_i}\Delta_j^{A_i}\\
    &\cong \Delta_j^{A_i}.
    \end{align*}
    \item Suppose $Q$ is the quiver 
     \begin{center}
    % https://q.uiver.app/#q=WzAsMyxbMCwwLCIxIl0sWzEsMCwiMyJdLFsyLDAsIjIiXSxbMCwxLCJcXGFscGhhIl0sWzEsMiwiXFxiZXRhIl0sWzAsMiwiXFxnYW1tYSIsMCx7ImN1cnZlIjotM31dXQ==
\begin{tikzcd}[ampersand replacement=\&]
	2 \& 3 \& 1
	\arrow["\alpha", from=1-1, to=1-2]
	\arrow["\gamma", curve={height=-18pt}, from=1-1, to=1-3]
	\arrow["\beta", from=1-2, to=1-3]
\end{tikzcd} 
  \end{center}
  and let $A_1=A_2=A_3=\field Q'$ with the natural order, where $Q'$ is the $\textup{A}_2$ quiver 
  % https://q.uiver.app/#q=WzAsMixbMCwwLCIxIl0sWzEsMCwiMiJdLFswLDFdXQ==
\[\begin{tikzcd}[ampersand replacement=\&]
	1 \& 2
	\arrow[from=1-1, to=1-2]
\end{tikzcd}\]
and let $M_\alpha=M_\beta=M_\gamma=\field Q'$ as $\field Q'$-$\field Q'$-bimodule.
Then 
 \begin{align*}
    \Delta_{21}
    &\cong \bigoplus_{p \notin \field Q\sum_{i'>2}e_{i'}\field Q, s(p)=2}M_p\otimes_{A_2}\Delta_1^{A_2}\\
    &=A_2\otimes_{A_2}\Delta_1^{A_2} \oplus M_\gamma \otimes_{A_2}\Delta_1^{A_2}\\
    &=\begin{pmatrix}
        0 & M_\gamma \otimes_{A_2}L_1^{A_2} & 0\\
        0 & L_1^{A_2} & 0\\
        0 & 0 & 0
    \end{pmatrix}=\begin{pmatrix}
        0 & L_1^{Q'} & 0\\
        0 & L_1^{Q'} & 0\\
        0 & 0 & 0
    \end{pmatrix}
     \end{align*}
     where the above matrix rings are equipped with the $A$-multiplication given by identifying $A$ with the matrix ring
     $$\begin{pmatrix}
        A_1 & M_\gamma\oplus M_{\beta}\otimes M_{\alpha}  & M_{\beta}\\
        0 & A_2 & 0\\
        0 & M_{\alpha} & A_3
    \end{pmatrix}$$
    and considering the module structure given by matrix multiplication, i.e.
    \begin{align*}
    \begin{pmatrix}
        a_1 & m_\gamma+ m'_{\beta}\otimes m'_{\alpha}  & m_{\beta}\\
        0 & a_2 & 0\\
        0 & m_{\alpha} & a_3
    \end{pmatrix}\cdot 
        \begin{pmatrix}
        0 & x & 0\\
        0 & y & 0\\
        0 & 0 & 0
    \end{pmatrix}= \begin{pmatrix}
        0 & a_1x+m_{\gamma}y & 0\\
        0 & a_2y & 0\\
        0 & 0 & 0
    \end{pmatrix}.
    \end{align*}
 \end{enumerate}
 \end{example}
 The following theorem is in some ways a generalization and in some ways a specialization of \cite[Theorem 3.1]{triangular_matrix}. On the one hand, a triangular matrix ring is a species on an $\textup{A}_2$-quiver, on the other hand, we do not obtain an if and only if statement, and the conditions on the modules $M_\alpha$ here are stronger.
\begin{theorem}\label{thm_qh_species}
Suppose that for all arrows $\alpha$ in $Q$, $M_{\alpha}$ is projective as a right $A_{s(\alpha)}$-module whenever there is a path $p$ with $\max(p)=s(p)$ such that $\alpha$ appears in $p$, and projective as a left $A_{t(\alpha)}$-module whenever there is a path $p$ with $\max(p)= t(p)$ such that $\alpha$ appears in $p$.
  Then  $A$ is quasi-hereditary.
\end{theorem}

\begin{proof}
    We proceed by induction on the number $N=\sum_{i=1}^n m_i$ of primitive orthogonal idempotents $e_{ij}$ of $A$. If $N=1$, then $A=A_1$ is quasi-hereditary by assumption.\\
    Now suppose $N>1$, and consider the ideal $J=Ae_{nm_n}A$. Then, as $A$-$A$-bimodules
    \begin{align*}
        J=Ae_{nm_n}A=\left(\bigoplus_{s(p)=n}M_p\right)e_{nm_n}\left(\bigoplus_{t(p)=n}M_p\right).
    \end{align*}
    Note that for all paths $p$ with $s(p)=n$ we have $\max(p)=s(p)$, and for all $p$ with $t(p)=n$ we have $\max(p)= t(p)$. Thus, by assumption $M_\alpha$ is projective as a right $A_{s(\alpha)}$-module respectively as a left $A_{t(\alpha)}$-module for all $\alpha$ appearing in such paths. Hence $M_p$ is a tensor product of right-projective modules for every path $p$ in $Q$ with $s(p)=n$, and thus projective as a right $A_n$-module. Similarly, $M_p$ is a tensor product of left-projective modules for every  every path $p$ in $Q$ with $t(p)=n$, and hence projective as a left $A_n$-module.
    Thus, we have 
    \begin{align*}
        J=Ae_{nm_n}A&=\left(\bigoplus_{s(p)=n}M_p\right)e_{nm_n}\left(\bigoplus_{t(p)=n}M_p\right)\\
        \cong \bigoplus_{s(p)=n}M_p\otimes_{A_n}A_ne_{m_n}^n A_n \otimes_{A_n} \bigoplus_{t(p)=n}M_p&= A1_{A_n}\otimes_{A_n}A_ne_{m_n}^n A_n \otimes_{A_n}1_{A_n}A,
    \end{align*}
   where $A1_{A_n}$ is a projective right $A_n$-module and $1_{A_n}A$ is a projective left $A_n$-module.
    By assumption, $A_n$ is quasi-hereditary with respect to the natural order on the idempotents, so that $A_n e_{m_n}^nA_n$ is projective as a right $A_n$-module. But then $J\cong A1_{A_n}\otimes_{A_n}A_ne_{m_n}^n A_n \otimes_{A_n}1_{A_n}A$ is projective as a right $A$-module.\\
    By definition, $J^2=J$. Moreover, 
    \begin{align*}
        J\rad(A)J&=Ae_{nm_n}\rad(A)e_{nm_n}A\\
        &=Ae_{nm_n}(\bigoplus_{i=1}^n \rad(A_i)\oplus \bigoplus_{p\neq e_i\textup{ for }1\leq i\leq n}M_p)e_{nm_n}A\\
        &=\sum_{i=1}^n Ae_{nm_n}\rad(A_i)e_{nm_n}A+\sum_{p\neq e_i\textup{ for }1\leq i\leq n}Ae_{nm_n}M_pe_{nm_n}A\\
        &=Ae_{nm_n}\rad(A_n)e_{nm_n}A
        =(0),
    \end{align*}
    since $Q$ has no directed cycles and $A_n$ is quasi-hereditary.\\
    If $A_n\neq A_n e_{nm_n}A_n$, then $A/J$ is isomorphic to the tensor algebra of the species on $Q$ given by $((A'_i), (M'_\alpha)_\alpha)$, where $A_i'=A_i$ for $i\neq n$ and $A_n':=A_n/A_n e_{n m_n}A_n$, $M_{\alpha}':=M_{\alpha}$ for $s(\alpha)\neq n\neq t(\alpha)$, $M_\alpha':=M_{\alpha}/M_{\alpha}e_{n m_n}A_n$ for $s(\alpha)=n$ and $M_{\alpha}':=M_{\alpha}/A_ne_{nm_n}M_{\alpha}$ for $t(\alpha)=n$.
    Note that $A_i'$ is quasi-hereditary for all $1\leq i\leq n$ and $M_{\alpha}'$ is projective as a right $A_{s(\alpha)}$-module whenever $M_{\alpha}$ is projective as a right $A_{s(\alpha)}$-module, unless $t(\alpha)=n$; and projective as a left $A_{t(\alpha)}$-module whenever $M_{\alpha}$ is projective as a left $A_{t(\alpha)}$-module, unless $s(\alpha)=n$. 
    Since $n$ is the maximal vertex and $Q$ contains no oriented cycles, this implies that $M_{\alpha}$ is projective as a right $A_{s(\alpha)}$-module whenever $\alpha$ appears in a path $p$ with $\max(p)=s(p)$, and projective as a left $A_{t(\alpha)}$-module whenever $\alpha$ appears in a path $p$ with $\max(p)=t(p)$. Moreover, the number of primitive orthogonal idempotents of $A/J$ is $N-1$. Thus, by induction hypothesis, $A/J$ is quasi-hereditary.\\
    On the other hand, if $A_n=A_ne_{nm_n}A_n$, then $A/J$ is isomorphic to the tensor algebra of the species on $Q'$ given by $((A_i)_{1\leq i\leq n-1}, (M_{\alpha})_{\alpha\in E_{Q'}})$, where $Q'$ is the quiver obtained from $Q$ by deleting the vertex $n$ and all adjacent edges. In this case, the number of primitive orthogonal idempotents of $A/J$ is again $N-1$, and the species $Q'$ fulfills the assumptions of the theorem, so that $A/J$ is quasi-hereditary.
 \end{proof}
 \begin{remark}\label{remark_qh}
     Suppose instead that $Q$ is directed. Then, as we have seen in Example \ref{example_std}, the standard modules of $A$ are just the standard modules of the $A_i$ viewed as modules over $A$. In particular, it is easy to see, arguing analogously to \cite[Theorem 3.1]{triangular_matrix}, that $A$ has a filtration by standard modules if and only if $M_p$ is standardly filtered as a left $A_{t(p)}$-module for all paths $p$ in $Q$. In particular, since a triangular matrix ring is nothing but a species on the $\textup{A}_2$-quiver, this gives rise to a proper generalization of \cite[Theorem 3.1]{triangular_matrix}.
 \end{remark}
 \begin{example}
    Note that in Remark \ref{remark_qh} we require that $M_p$ is left standardly filtered for every path, not just for every arrow. This is because, for two composable arrows $\alpha:i\rightarrow j$ and $\beta:j\rightarrow k$ in $Q$, the fact that $M_\beta$ is left-standardly filtered in general does not imply that $M_\beta\otimes M_\alpha$ is. The following is an explicit example of this behaviour.\\
      Let $Q$ be the $\textup{A}_3$ quiver
      \begin{center}
      % https://q.uiver.app/#q=WzAsMyxbMCwwLCIxIl0sWzEsMCwiMiJdLFsyLDAsIjMiXSxbMCwxLCJcXGFscGhhIl0sWzEsMiwiXFxiZXRhIl1d
\begin{tikzcd}[ampersand replacement=\&]
	1 \& 2 \& 3
	\arrow["\alpha", from=1-1, to=1-2]
	\arrow["\beta", from=1-2, to=1-3]
\end{tikzcd}
  \end{center}
      and let $Q'$ be the $\textup{A}_2$-quiver
      \begin{center}
      % https://q.uiver.app/#q=WzAsMixbMCwwLCIxIl0sWzEsMCwiMiJdLFswLDEsIlxcYWxwaGEiXV0=
\begin{tikzcd}[ampersand replacement=\&]
	1 \& 2
	\arrow["\alpha", from=1-1, to=1-2]
\end{tikzcd}
  \end{center}
        Let  $A_1=\field Q'=A_2$ be the quasi-hereditary algebra equipped with the natural order, so that $\Delta_1^{A_i}=L_1^{A_i}$ and $\Delta_2^{A_i}=P_2^{A_i}=L_2^{A_i}$ for $i=1,2$.\\
        Let $A_3=\field Q'$ be the quasi-hereditary algebra equipped with the opposite of the natural order, so that $\Delta_3^{A_3}=P_1^{A_3}$ and $\Delta_2^{A_3}=L_2^{A_3}=P_2^{A_3}$.
        Let $M_\beta=\field Q'$ as an $A_3$-$A_2$-bimodule, and let $M_{\alpha}=\field e_1$ as an $A_2$-$A_1$-bimodule, that is, $M_\alpha\cong L_1^{A_2}\otimes L_1^{A_1^{\op}}$.\\
        Let $A$ be the tensor algebra of the species on $Q$ given by $( (A_i)_i, (M_{\alpha}, M_{\beta}))$ with the natural order on $Q$. Then, since $\field Q$ is directed, the standard modules of $A$ are simply the standard modules of $A_1, A_2$ and $A_3$ viewed as $A$-modules.\\
      Despite the fact that $M_{\alpha}$ and $M_\beta$ are standardly filtered as left modules, $A$ doesn't have a standard filtration since $M_{\beta\alpha}\cong L_1^{A_3}\otimes L_1^{A_1^{\op}}$ isn't standardly filtered as a left $A_3$-module.
      \end{example}
\subsection{Exact Borel subalgebras of species}
In the following, we investigate when exact Borel subalgebras of $A_1, \dots, A_n$ give rise to an exact Borel subalgebra of $A$.
We begin by fixing one last piece of notation:
\begin{definition}\label{def_tensor}
    Let $(A, \leq_A)$ be a quasi-hereditary algebra and $B\subseteq A$ be an exact Borel subalgebra of $A$. Then for any $B$-module $N$ we denote by $f_N$ the $B$-module homorphism
    \begin{align*}
        f_N:N\rightarrow A\otimes_B N, n\mapsto 1\otimes n.
    \end{align*}
    Note that e.g. by \cite[Theorem 3.1]{Conde2}, $f_N$ is a monomorphism
\end{definition}
Now we are ready to prove the main theorem of this section. The goal is to use an exact Borel subalgebra $B_Q=\field H$ of the underlying quiver $Q$ of the species, as well collection of exact Borel subalgebras $(B_{i})_i$ of the quasi-hereditary algebras $(A_i)_i$ in order to construct a Borel subalgebra of the tensor algebra of the species via assembling assembling the $(B_{i})_i$ together with an appropriate set of bimodules into a species on $H$. In order to be able to define the right $B_i$-$B_j$-bimodules, we need to assume some conditions on the bimodules $(M_{\alpha})_{\alpha}$.
 \begin{theorem}\label{thm_borel_species}
 Suppose we are in the setting of Theorem \ref{thm_qh_species}. Then \cite[Theorem 3.6]{monomialborel} yields an exact Borel subalgebra $B_Q$ of $\field Q$ with a basis $\mathcal{B}$ given by paths in $\field Q$. Let $Q'$ be the quiver with vertices $1,\dots ,n$ and arrows given by the elements of $\mathcal{B}$.\\
     Suppose additionally that for every $1\leq i\leq n$ the quasi-hereditary algebra $A_i$ has an exact Borel subalgebra $B_i$, and assume that for every $\beta: i\rightarrow j\in \mathcal{B}$ there is a $B_j$-$B_i$-bimodule $N_\beta$ such that $M_{\beta}\cong A_j\otimes_{B_j} N_{\beta}$ as $A_j$-$B_i$-bimodules. Denote the isomorphism by $\phi_{\beta}:  A_j\otimes_{B_j} N_{\beta}\rightarrow M_{\beta}$.\\ 
     Then the tensor algebra $B$ of the species on $Q'$ given by $((B_i)_i, (N_{\beta})_{\beta\in \mathcal{B}})$ is an exact Borel subalgebra of $A$ with the embedding given componentwise by
     \begin{align*}
      \iota_{\beta_1\dots\beta_n}:  N_{\beta_1}\otimes_{B_2}N_{\beta_2}\otimes_{B_3}\dots\otimes_{B_{n}}N_{\beta_n}\rightarrow M_{\beta_1\dots\beta_n}, \\
        x_1\otimes\dots\otimes x_n\mapsto \phi_{\beta_1}(1_{A_1}\otimes x_1)\otimes_{A_2}\dots\otimes_{A_n}\phi_{\beta_n}(1_{A_n}\otimes x_n)
     \end{align*}
     respectively by the embeddings $\iota_i:B_i\rightarrow A_i$.
 \end{theorem}
 \begin{proof}
     For ease of notation, we will assume that $B_i$ is in fact a subset of $A_i$, that is, we will suppress the embeddings $\iota_i: B_i\rightarrow A_i$.\\
     Since all $B_i$ are directed and $Q'$ is directed, $B$ is directed. Moreover, note that $\iota$ is indeed injective, since it is a composition of the canonical embedding
     \begin{align*}
     f_{N_{\beta_1}\otimes_{B_2}N_{\beta_2}\otimes_{B_3}\dots\otimes_{B_{n}}N_{\beta_n}}:& N_{\beta_1}\otimes_{B_2}N_{\beta_2}\otimes_{B_3}\dots\otimes_{B_{n}}N_{\beta_n}\\&\rightarrow A_1\otimes_{B_1}N_1\otimes_{B_2}N_{\beta_2}\otimes_{B_3}\dots\otimes_{B_{n}}N_{\beta_n}
     \end{align*}
     from Lemma \ref{def_tensor} with the isomorphism $\phi_p$ given by the composition
   % https://q.uiver.app/#q=WzAsNixbMCwwLCJBXzFcXG90aW1lc197Ql8xfU5fMVxcb3RpbWVzX3tCXzJ9Tl97XFxiZXRhXzJ9XFxvdGltZXNfe0JfM31cXGRvdHNcXG90aW1lc197Ql97bn19Tl97XFxiZXRhX259Il0sWzEsMCwiTV97XFxiZXRhXzF9XFxvdGltZXNfe0JfMn1OX3tcXGJldGFfMn1cXG90aW1lc197Ql8zfVxcZG90c1xcb3RpbWVzX3tCX3tufX1OX3tcXGJldGFfbn0iXSxbMSwxLCJNX3tcXGJldGFfMX1cXG90aW1lc197QV8yfUFfMlxcb3RpbWVzX3tCXzJ9Tl97XFxiZXRhXzJ9XFxvdGltZXNfe0JfM31cXGRvdHNcXG90aW1lc197Ql97bn19Tl97XFxiZXRhX259Il0sWzAsMSwiTV97XFxiZXRhXzF9XFxvdGltZXNfe0FfMn1NX3tcXGJldGFfMn1cXG90aW1lc197Ql8zfVxcZG90c1xcb3RpbWVzX3tCX3tufX1OX3tcXGJldGFfbn0iXSxbMCwyLCJcXGRvdHMiXSxbMSwyLCJNX3tcXGJldGFfMVxcZG90c1xcYmV0YV9ufT1NX3tcXGJldGFfMX1cXG90aW1lc197QV8yfVxcZG90c1xcb3RpbWVzX3tBX259IE1fe1xcYmV0YV9ufSJdLFswLDEsIlxccGhpX3tcXGJldGFfMX1cXG90aW1lc1xcaWQiXSxbMSwyLCJcXHNpbSJdLFsyLDMsIlxcaWRcXG90aW1lc1xccGhpX3tcXGJldGFfMn1cXG90aW1lc1xcaWQiLDJdLFszLDQsIlxcc2ltIiwyXSxbNCw1LCJcXHNpbSIsMl1d
\[\begin{tikzcd}[ampersand replacement=\&]
	{A_1\otimes_{B_1}N_1\otimes_{B_2}N_{\beta_2}\otimes_{B_3}\dots\otimes_{B_{n}}N_{\beta_n}} \& {M_{\beta_1}\otimes_{B_2}N_{\beta_2}\otimes_{B_3}\dots\otimes_{B_{n}}N_{\beta_n}} \\
	{M_{\beta_1}\otimes_{A_2}M_{\beta_2}\otimes_{B_3}\dots\otimes_{B_{n}}N_{\beta_n}} \& {M_{\beta_1}\otimes_{A_2}A_2\otimes_{B_2}N_{\beta_2}\otimes_{B_3}\dots\otimes_{B_{n}}N_{\beta_n}} \\
	\dots \& {M_{\beta_1\dots\beta_n}=M_{\beta_1}\otimes_{A_2}\dots\otimes_{A_n} M_{\beta_n}}
	\arrow["{\phi_{\beta_1}\otimes\id}", from=1-1, to=1-2]
	\arrow["\sim", from=1-2, to=2-2]
	\arrow["\sim"', from=2-1, to=3-1]
	\arrow["{\id\otimes\phi_{\beta_2}\otimes\id}"', from=2-2, to=2-1]
	\arrow["\sim"', from=3-1, to=3-2]
\end{tikzcd}\]
Let us first show that $A$ is projective as a right $B$-module. By \cite[Corollary 3.11]{monomialborel}, there is a set $\mathcal{R}$ of paths  in $\field Q$ such that there is an isomorphism of right $B_Q$-modules 
\begin{align*}
    \bigoplus_{r\in \mathcal{R}}e_{s(r)}B_Q\rightarrow \field Q,\; (e_{s(r)}b)_{r\in \mathcal{R}}\mapsto \sum_{r\in \mathcal{R}}rb.
\end{align*}
In particular, for any path $p$ in $Q$ there is exactly one $r\in \mathcal{R}$ and one path $q$ in $Q'$ such that $p=rq$.
Now consider the map
\begin{align*}
   \phi: \bigoplus_{r\in \mathcal{R}}A_{s(r)}\otimes_{B_{s(r)}}1_{B_{s(r)}}B=\bigoplus_{r\in \mathcal{R}}A_{s(r)}\otimes_{B_{s(r)}} (\bigoplus_{q\textup{ path in }Q', t(q)=s(r)}N_q)\rightarrow A\\
    (a_{r}\otimes n_q)_{p,q}\mapsto (\phi_{rq}(a_r\otimes n_q))_{rq}.
\end{align*}
Then, since for every $p$ there is exactly one pair $(r,q)$ such that $p=rq$, and since for every pair $(r,q)$ the map $\phi_{rq}$ is an isomorphism, $\phi$ is an isomorphism. Moreover, $\phi$ is a right $B$-module homomorphism and $\bigoplus_{r\in \mathcal{R}}A_{s(r)}\otimes_{B_{s(r)}}1_{B_{s(r)}}B$ is projective as a right $B$-module, since $A_{s(r)}$ is projective as a right $B_{s(r)}$-module and $1_{B_{s(r)}}B$ is projective as a right $B$-module for every $r\in \mathcal{R}$. Hence, $A$ is projective as a right $B$-module.\\
Recall that the set of isomorphism classes of simple $B$-modules is in bijection with the disjoint union of the sets of isomorphism classes of simple $B_i$-modules, and the set of isomorphism classes of simple $A$-modules is in bijection with the disjoint union of the sets of isomorphism classes of simple $A_i$-modules. Hence the bijections $\varphi_i:\Sim(B_i)\rightarrow \Sim(A_i)$ induce a bijection $\varphi:\Sim(B)\rightarrow \Sim(A)$. Let $L_{ij}^B$ be a simple $B$-module corresponding to a simple $B_i$-module $L_j^{B_i}$, that is, $L_{ij}^B\cong B_i^B\otimes_{B_i}L_j^{B_i}$.
Note that
\begin{align*}
   &A\otimes_B B_i^B\cong A1_{A_i}/\iota(A\otimes \bigoplus_{e_i\neq p\in \mathcal{B}}N_p )\\
   =&\bigoplus_{s(p)=i}M_p/\bigoplus_{e_i\neq p\in \mathcal{B}, s(p)=i, q}M_{qp}\cong \bigoplus_{s(p)=i, p\notin \field Q(\mathcal{B}\setminus \{e_i\})}M_p.
\end{align*}
Note that $\field Q(\mathcal{B}\setminus \{e_i\})e_i=\field Q\rad(B)e_i=\ker(\pi_i)$, where $\pi_i: P_i^Q\rightarrow \Delta_i^Q$ is the canonical projection, since $B_{Q}$ is an exact Borel subalgebra of $\field Q$. In particular, $\field Q(\mathcal{B}\setminus \{e_i\})e_i= \field Q \sum_{i'>i}e_{i'}\field Qe_i$, so that the module above is isomorphic to
\begin{align*}
   \bigoplus_{s(p)=i, p\notin \field Q \sum_{i'>i}e_{i'}\field Qe_i}M_p.
\end{align*}
Since $L_{ij}^B\cong B_i^B\otimes_{B_i}L_j^{B_i}$ we thus have
\begin{align*}
    A\otimes_B L_{ij}^{B}&\cong A\otimes_B B_i\otimes_{B_i} L_j^{B_i}\\
    &\cong   \bigoplus_{s(p)=i, p\notin \field Q \sum_{i'>i}e_{i'}\field Qe_i}M_p\otimes_{B_i} L_j^{B_i} \\
    &\cong   \bigoplus_{s(p)=i, p\notin \field Q \sum_{i'>i}e_{i'}\field Qe_i}M_p\otimes_{A_i} A_i\otimes_{B_i} L_j^{B_i}\\
    &\cong \bigoplus_{s(p)=i, p\notin \field Q \sum_{i'>i}e_{i'}\field Qe_i}M_p\otimes_{A_i}\Delta(\varphi_i(L_j^{B_i})^{A_i}\cong \Delta(\varphi(L_{ij}^B)).
\end{align*}
 \end{proof}
 \begin{corollary}\label{corollary_example}
     \begin{enumerate}
         \item Suppose $A_i=A_1$ for all $1\leq i\leq n$, $A_1$ has an exact Borel subalgebra $B_1$ and $M_{\alpha}\cong A_1$  for every $\alpha$. Then the tensor algebra $B$ of the species on $Q'$ given by $((B_1)_i, (B_1)_{\alpha})$ is an exact Borel subalgebra.
         In fact, in this case $A\cong A_1\otimes \field Q$, and under this isomorphism $B$ corresponds to the exact Borel subalgebra $B_1\otimes B_Q$ from Theorem \ref{thm_tensor}.
         \item Suppose $Q$ is directed. Then, whenever $M_{\alpha}$ is left projective for all $\alpha$, $((V_Q, \emptyset), (B_i)_i, \emptyset)$ is an exact Borel subalgebra.
         \item Suppose $M_{\alpha}$ is projective as an $A_i$-$A_j$-bimodule for all $\alpha$. Then $M_{\alpha}\cong P_{\alpha}\otimes_{\field}Q_{\alpha}$ for some projective $A_i$-module $P_{\alpha}$ and some projective $A_j^{\op}$-module $Q_{\alpha}$; and since $B_{t(\alpha)}$ is a regular exact Borel subalgebra of $A_{t(\alpha)}$ and the projective modules have a standard filtration, $P_{\alpha}\cong A_{t(\alpha)}\otimes_{B_{t(\alpha)}}N'_{\alpha}$ for some $B_{t(\alpha)}$-module $N'_{\alpha}$ \cite[A.4, p. 34]{KKO}. Hence the tensor algebra of the species on $Q'$ given by
         $((B_i)_i, (N'_{\alpha}\otimes Q_{\alpha}))$ is an exact Borel subalgebra of $A$.\\
         In the special case where $M_{\alpha}\cong A_{t(\alpha)}\otimes_{\field}A_{s(\alpha)}$ such that $A$ is a generalized path algebra in the sense of \cite{coelholiu} we obtain that the tensor algebra $B$ of the species on $Q'$ given by
         $((B_i)_i, (B_{t(\alpha)}\otimes A_{s(\alpha)}))$ is an exact Borel subalgebra of $A$. Note that this is in general not  a generalized path algebra in the sense of \cite{coelholiu}. 
     \end{enumerate}
 \end{corollary}
 The following proposition is the corresponding statement for triangular matrix rings under appropriately weakened assumptions. The proof, however, is essentially the same as the proof of the theorem above.
 \begin{proposition}\label{proposition_triangular_borel}
    Suppose $A_1$ and $A_2$ are quasi-hereditary algebra and $M$ is a left standardly filtered $A_2$-$A_1$-bimodule. Let $A$ be the triangular matrix ring 
\begin{align*}
    \begin{pmatrix}
        A_2 & M\\
        0 & A_1
    \end{pmatrix}
\end{align*}
    Suppose that $B_1$ and $B_2$ are exact Borel subalgebras of $A_1$ resp. $A_2$ with embeddings $\iota_1$ resp. $\iota_2$, and $M\cong A_2\otimes_{B_2} N$ as an $A_2$-$B_1$-bimodule for some $N_2$-$N_1$-bimodule $B$. Then the triangular matrix ring $B$ given by
    \begin{align*}
    \begin{pmatrix}
        B_2 & N\\
        0 & B_1
    \end{pmatrix}
\end{align*}
is an exact Borel subalgebra of $A$ via the embedding
     \begin{align*}
           \begin{pmatrix}
        \iota_2 & f_N\\
        0 & \iota_1
    \end{pmatrix}: B\rightarrow A.
     \end{align*}
 \end{proposition}
 \begin{proof}
     It is clear that \begin{align*}
           \begin{pmatrix}
        \iota_2 & f_N\\
        0 & \iota_1
    \end{pmatrix}: B\rightarrow A.
     \end{align*} is an embedding. 
     Moreover, as right $B$-modules
     \begin{align*}
         A=\begin{pmatrix}
             0 & 0 \\
             0 & A_1
         \end{pmatrix}\oplus \begin{pmatrix}
             A_2 & M \\
             0 & 0
         \end{pmatrix}
         \cong \begin{pmatrix}
             0 & 0 \\
             0 & A_1\otimes_{B_1} B_1
         \end{pmatrix}\oplus \begin{pmatrix}
             A_2\otimes_{B_2} B_2 & A_2\otimes_{B_2} N \\
             0 & 0
         \end{pmatrix}\\
         \cong A_1\otimes_{B_1} \begin{pmatrix}
             0 & 0 \\
             0 &  B_1
         \end{pmatrix}\oplus A_2\otimes_{B_2} \begin{pmatrix}
             B_2 & N \\
             0 &  0
         \end{pmatrix}
         =
         A_1\otimes_{B_1} 1_{B_1}B\oplus A_2\otimes_{B_2} 1_{B_2}B
     \end{align*}
   so that $A$ is projective as a right $B$-module. \\
   As in Theorem \ref{thm_borel_species}, let $\varphi: \Sim(B)\rightarrow \Sim(A)$ be the bijection coming from the bijections $\varphi_i:\Sim(B_i)\rightarrow \Sim(A_i)$ for $1\leq i\leq 2$.
     It follows as in Theorem \ref{thm_borel_species} that the simple $B$-modules induce the corresponding standard modules over $A$.
 \end{proof}
 \begin{comment}
 \begin{example}\label{example_triangular}(Triangular Matrix Rings)\\
         It was shown in \cite[Theorem 3.1]{triangular_matrix} that a triangular matrix ring $\begin{pmatrix}
             A & M\\ 0&A' 
         \end{pmatrix}$ of quasi-hereditary algebras $A$, $A'$ and an $A$-$A'$-bimodule $M$ with ${}_{A}M\in \filt(\Delta^A)$ is quasi-hereditary with the order of the simples given by extending the previous orders by $L_{i}^{A'}<L_j^A$ for all $i,j$. Note that a triangular matrix ring, in our notation, is nothing but the species given by $(\textup{A}_2, (A_i)_{i=1,2}, (M)_{\alpha})$ where $A_1=A'$, $A_2=A$ and $\textup{A}_2$ is the quiver
         % https://q.uiver.app/#q=WzAsMixbMCwwLCIxIl0sWzEsMCwiMiJdLFswLDEsIlxcYWxwaGEiXV0=
\begin{tikzcd}[ampersand replacement=\&]
	1 \& 2
	\arrow["\alpha", from=1-1, to=1-2]
\end{tikzcd}
with the order $1<2$. Hence we re-obtain this result in the case that $M$ is projective as a left $A$-module.\\
What is more, given the result in \cite[Theorem 3.1]{triangular_matrix}, it can be seen easily from our proof in \ref{thm_borel_species} that in the case of a triangular matrix ring, if $A$ and $A'$ are quasi-hereditary with exact Borel subalgebra $B$ and $B'$ respectively, and if  ${}_{A}M\in \filt(\Delta^A)$ such that $M\cong A\otimes_B M$ as an $A$-$B'$-bimodule, then 
\begin{align*}
    \begin{pmatrix}
        B & N \\ 0 & B'
    \end{pmatrix}
\end{align*}
is an exact Borel subalgebra of $\begin{pmatrix}
             A & M\\ 0&A' 
         \end{pmatrix}$ 
     \end{example}
\end{comment}
      
      \subsection{Regularity}
      Given Theorem \ref{thm_borel_species}, it is a natural question under which conditions the exact Borel subalgebra constructed therein is regular. This turns out to be more complicated than the analogous question for tensor products, which we answered in \ref{proposition_regular}. Instead of giving a complete characterization, we therefore only give some examples, which hopefully illustrate, that, similarly to the case for tensor algebras, regularity cannot be expected in most cases.\\[0.7cm]
       Suppose we have a species on $Q$ given by $((A_i)_i, (M_{\alpha})_{\alpha})$ where $A_i=A_1$ for all $i$ and $M_{\alpha}=A_1$ as $A_1$-$A_1$-bimodules for all $\alpha$. Then, as we have seen in Corollary \ref{corollary_example}, the tensor algebra $A$ of $((A_i)_i, (M_{\alpha})_{\alpha})$ is isomorphic to $A_1\otimes \field Q$ as algebras, and if  $A_1$ has an exact Borel subalgebra $B_1$, then the exact Borel subalgebra in $A$ that we constructed in Theorem \ref{thm_borel_species} corresponds to $B_1\otimes B_Q$. Thus, even assuming that $B_1$ and $B_Q$ are regular, it is regular only under very restrictive conditions, as was shown in Proposition \ref{proposition_regular}. In general, we suppose that is difficult to conclude regularity of the exact Borel subalgebra for the species as well. However, similarly to the case for tensor algebras there are still some degenerate cases where the constructed Borel subalgebra is indeed general.
      \begin{remark}\label{proposition_regular_species}
      \begin{enumerate}
          \item Suppose $\field Q$ is directed and every $A_i$ is directed. Then, $A$ is directed, so the only exact Borel subalgebras of $A$ are its maximal semisimple subalgebras, and they are necessarily regular.
          \item Suppose $\field Q^{\op}$ is directed and every $A_i^{\op}$ is directed. Then, $A^{\op}$ is directed, so the only exact Borel subalgebra of $A$ is $A$ and it is necessarily regular.
          \item Suppose $\field Q$ is semisimple. Then, $A=\bigoplus_{i=1}^n A_i$, so that if for every $1\leq i\leq n$, $B_i$ is a regular exact Borel subalgebra of $A_i$, $B:=\bigoplus_{i=1}^n B_i$ is a regular exact Borel subalgebra of $A$.
          \end{enumerate}
          \end{remark}
          In analogy to Proposition \ref{proposition_regular}, one may expect regularity also in the case where all of the $A_i$ are semisimple. However, this does not hold in general.
          \begin{example} Suppose $A_i$ is semisimple for every $1\leq i\leq n$. Then, even if $\field Q$ has a regular exact Borel subalgebra, this is not necessarily true for $A$. To see this, let $Q$ be the $\textup{A}_3$ quiver 
          % https://q.uiver.app/#q=WzAsMyxbMCwwLCIxIl0sWzIsMCwiMyJdLFsxLDAsIjIiXSxbMiwxLCJcXGJldGEiXSxbMCwyLCJcXGFscGhhIl1d
\[\begin{tikzcd}[ampersand replacement=\&]
	3 \& 1 \& 2,
	\arrow["\alpha", from=1-1, to=1-2]
	\arrow["\beta", from=1-2, to=1-3]
\end{tikzcd}\]
let $A_2=\field=A_3$ and $A_1=\field^2$, and let $M_\alpha=A_1$ as a right $A_1$-module and $M_\beta=A_1$ as a left $A_1$-module. Note that by \cite[Theorem 60]{Markus}, $\field Q$ admits a regular exact Borel subalgebra.
However, $A$ is isomorphic to the path algebra of the quiver $Q'$ given by  % https://q.uiver.app/#q=WzAsNCxbMCwxLCI0Il0sWzEsMCwiMSJdLFsxLDIsIjIiXSxbMiwxLCIzIl0sWzEsMywiXFxkZWx0YSJdLFsyLDMsIlxcZ2FtbWEiLDJdLFswLDEsIlxcYWxwaGEiXSxbMCwyLCJcXGJldGEiLDJdXQ==
\[\begin{tikzcd}[ampersand replacement=\&]
	\& 1 \\
	4 \&\& 3 \\
	\& 2
	\arrow["\delta", from=1-2, to=2-3]
	\arrow["\alpha", from=2-1, to=1-2]
	\arrow["\beta"', from=2-1, to=3-2]
	\arrow["\gamma"', from=3-2, to=2-3]
\end{tikzcd}\]
      with the natural order on the vertices. By \cite[Example 5.17 (3)]{monomialborel}, this has an exact Borel subalgebra $B$ which is not regular. However, since $\field Q'$ is basic, any regular exact Borel subalgebra of $\field Q'$ would  be conjugate to $B$ by \cite[Theorem 8.4]{uniqueness}. Since $B$ is not regular, this implies that  $\field Q'$, and hence $A$, does not admit a regular exact Borel subalgebra.
      \end{example}
      Finally, we would like to remark that in contrast to the case of tensor algebras, there are many cases besides where the constructed exact Borel subalgebra is regular, seemingly by chance:
     \begin{example}
          Let $Q$ be the $\textup{A}_2$-quiver
      % https://q.uiver.app/#q=WzAsMixbMCwwLCIxIl0sWzEsMCwiMiJdLFswLDEsIlxcYWxwaGEiXV0=
\[\begin{tikzcd}[ampersand replacement=\&]
	1 \& 2,
	\arrow["\alpha", from=1-1, to=1-2]
\end{tikzcd}\]
let $A_1=A_2=\field Q$, where $A_2$ has the usual order and $A_1$ the opposite of the usual order, and let $M_\alpha=\field \alpha$ viewed as an $A_2$-$A_1$-bimodule.\\
Then $M_\alpha\cong L_2^{A_2}=P_2^{A_2}$ as a left $A_2$-module and $M_\alpha\cong L_1^{A_1^{\op}}=P_1^{A_1^{\op}}$ as a right $A_1$-module. Moreover, $A$ is isomorphic to the path algebra of the quiver $Q'$
% https://q.uiver.app/#q=WzAsNCxbMCwxLCIxIl0sWzAsMCwiMiJdLFsxLDAsIjMiXSxbMSwxLCI0Il0sWzEsMF0sWzIsM10sWzEsM11d
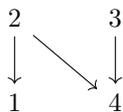
\[\begin{tikzcd}[ampersand replacement=\&]
	2 \& 3 \\
	1 \& 4
	\arrow[from=1-1, to=2-1]
	\arrow[from=1-1, to=2-2]
	\arrow[from=1-2, to=2-2]
\end{tikzcd}\]
with the natural order. By \cite[Example 4.13]{Conde}, this has a regular exact Borel subalgebra.
     \end{example}
     
\bibliography{speciestest}
\end{document}